\documentclass[11pt, draft]{amsart}[2000]

\usepackage{amsmath}
\usepackage{amssymb}
\usepackage{amsthm}  
\usepackage[all]{xypic}
\usepackage{enumitem}
\usepackage{mathtools}
\usepackage{color}
\usepackage{bbm}
\usepackage{tikz-cd}
\usepackage{soul}
\usepackage[all]{xy}
\usepackage{hyperref}

\usepackage[hmargin=2.5cm,vmargin=2cm]{geometry}

\theoremstyle{plain}

\newtheorem{Thm}{Theorem}[section]
\newtheorem{Cor}[Thm]{Corollary}
\newtheorem{Lemma}[Thm]{Lemma}

\newtheorem{Prop}[Thm]{Proposition}

\theoremstyle{definition}
\newtheorem{Def}[Thm]{Definition}

\newtheorem{Exl}[Thm]{Example}
\newtheorem{Rmk}[Thm]{Remark}

\numberwithin{equation}{section}

\renewcommand{\O}{{\mathcal{O}}}

\newcommand{\Aut}{\operatorname{Aut}}

\newcommand{\Gr}{\operatorname{Gr}}

\newcommand{\id}{{\operatorname{id}}}

\newcommand{\Span}{\operatorname{span}}

\DeclareMathOperator{\supp}{supp}

\DeclareMathOperator{\Prob}{Prob}

\begin{document}
	
	\title[Ratio-limit boundaries of relatively hyperbolic groups]{Ratio-limit boundaries for random walks on relatively hyperbolic groups}
	
	\author[A. Dor-On]{Adam Dor-On}
	\address{Department of Mathematics \\ University of Haifa \\ Haifa \\ Israel.}
	\email{adoron.math@gmail.com}

        \author[M. Dussaule]{Matthieu Dussaule}
	\address{Cogitamus Laboratory \\ France.}
	\email{matthieu.dussaule@hotmail.fr}
	
	\author[I. Gekhtman]{Ilya Gekhtman}
	\address{Department of Mathematics \\ Technion - IIT \\ Haifa \\ Israel.}
	\email{ilyagekh@gmail.com}

	
	\thanks{The first author was partially supported by EU grant HORIZON-MSCA-SE-2021 Project 101086394, and by an ISF Moked 713510 grant number 2919/19 for his stay at the Weizmann Institute of Science, Israel.}

	\begin{abstract}
We study boundaries arising from limits of ratios of transition probabilities for random walks on relatively hyperbolic groups. We extend, as well as determine significant limitations of, a strategy employed by Woess for computing ratio-limit boundaries for the class of hyperbolic groups. On the one hand we employ results of the second and third authors to adapt this strategy to spectrally non-degenerate random walks, and show that the closure of minimal points in $R$-Martin boundary is the unique smallest invariant subspace in ratio-limit boundary. On the other hand we show that the general strategy can fail when the random walk is spectrally degenerate and adapted on a free product. Using our results, we are able to extend a theorem of the first author beyond the hyperbolic case and establish the existence of a co-universal quotient for Toeplitz C*-algebras arising from random walks which are spectrally non-degenerate on relatively hyperbolic groups.
Finally, we exhibit an example of a relatively hyperbolic group carrying two random walks such that the ratio limit boundaries are not equivariantly homeomorphic and no two equivariant quotients of their respective Toeplitz C*-algebras are equivariantly $*$-isomorphic.
	\end{abstract}
	
		
\maketitle

\renewcommand{\O}{{\mathcal{O}}}
	
\section{Introduction}

\subsection*{Background}

A prominent line of research in geometric group theory is dedicated to finding appropriate definitions of boundaries on which a group acts, following examples of Fuchsian groups acting on the hyperbolic plane. Probabilistic approaches for defining and studying such boundaries were initiated by Furstenberg \cite{Fur71, Fur72}, going back to ideas of Martin and Poisson in classical harmonic analysis.

These approaches use random walks to define probabilistic boundaries obtained by considering different ways a random walker can ``converge at infinity". By a result of Varopoulos \cite{Var85} we know that the only groups that support recurrent random walks are virtually $\mathbb{Z}^d$ for $d =0,1,2$.  Hence, most random walks are transient, and it makes sense to define probabilistic compactifications as the set of points to which a random walk converges.

Throughout this paper we will consider a countable discrete group $\Gamma$ together with a probability measure $\mu$ whose support generates $\Gamma$ as a semigroup. We will call such a probability measure \emph{admissible}. This defines a Markov chain on $\Gamma$ given by the transition kernel $P(g,h)=\mu(g^{-1}h)$ which is called the $\mu$-random walk on $\Gamma$.
We denote by $P^n(g,h)=\mu^{*n}(g^{-1}h)$ the associated $n$-step transition probabilities. 

The \emph{weighted Green function} associated to $\mu$ is defined by
$$G(g,h | r)=\sum^\infty_{n=0}P^{n}(g,h)r^n $$ for $r\geq 1$ and $g,h \in \Gamma$.
This sum converges when $r <R$ where
$$R:=R(\mu):= \big(\limsup_{n\rightarrow \infty}\sqrt[n]{P^{n}(e,e)}\big)^{-1}$$ is the inverse of the spectral radius of the random walk. When $\Gamma$ is non-amenable, we know that $R>1$ (See \cite{Kes59}), and $G(g,h | R)<\infty$ for all $g,h$ in $\Gamma$ (see \cite[Theorem~7.8]{Woe00}). Thus, it makes sense to define the \emph{$r$-Martin kernel} $K_r: \Gamma \times \Gamma \rightarrow (0,\infty)$ at $r \in [1,R]$ (which is also well-defined as $r\rightarrow R$) by
$$
K_r(g,h) := \frac{G(g,h |r)}{G(e,h|r)}.
$$ 
The smallest compactification of $\Gamma$ to which the functions $h \mapsto K_r(g,h)$ extend continuously for all $g\in \Gamma$ is called the $r$-Martin compactification, and the complement of $\Gamma$ in it is called the $r$-Martin boundary $\partial_{M,r} \Gamma$. 

Martin boundaries are intimately related to discrete harmonic analysis on the group through the Martin--Poisson integral representation theorem. We will say that a function $u : \Gamma \rightarrow (0,\infty)$ is $t$-harmonic if
$$t \cdot u(x) = \sum_{y\in \Gamma} P(x,y)u(y).$$ The set $\mathcal{H}_1^+(P,t)$ of positive harmonic functions with $u(e) =1$ is a compact convex set with the topology of pointwise convergence, and the functions $x\mapsto K_r(x,\xi)$ are $r^{-1}$-harmonic whenever $\xi \in \partial_{M,r}\Gamma$. We will denote by $\partial^m_{M,r}\Gamma$ the points $\xi \in \partial_{M,r}\Gamma$ for which the function $x\mapsto K_r(x,\xi)$ is extreme in $\mathcal{H}_1^+(P,r^{-1})$. These are often called \emph{minimal} points of $\mathcal{H}_1^+(P,r^{-1})$, and the following version of Choquet's theorem shows how to get any function in $\mathcal{H}_1^+(P,r^{-1})$ from minimal such.

\begin{Thm}[Poisson--Martin integral representation]
Let $r\in [1,R]$, and $u$ a positive $r^{-1}$-harmonic function with $u(e)=1$. Then there is a representing probability measure $\nu^u$ on $\partial_{M,r}\Gamma$ such that
$$
u(x) = \int_{\partial_{M,r}\Gamma} K_r(x,\xi) d\nu^u(\xi),
$$
and $\nu^u$ is unique among probability measures $\nu$ that have full mass on $\partial^m_{M,r}\Gamma$. 
\end{Thm}

There are several different notions of boundaries for discrete groups in the literature, and it is useful and interesting to know when they coincide or are different. For instance, the Martin boundary may \textit{a priori} depend on both $r$ and $\mu$, and it is useful to know when this is the case.
For hyperbolic groups, it was shown by Ancona \cite{Anc88} that for a large class of random walks on them, the $1$-Martin boundary can be identified with the Gromov boundary $\partial \Gamma$ of $\Gamma$. Ancona's results extend readily to the $r$-Martin boundary for $r<R$, and Gouezel \cite{Gou14} extended them to $r=R$ when $\mu$ is symmetric. This opens avenues for using random walks to study geometric properties of the group. Among several applications, Ancona inequalities found use in giving a new proof of the Baum-Connes conjecture for hyperbolic groups \cite{HM11}, and for constructing natural flow with unique KMS stats on the crossed product C*-algebra $C(\partial \Gamma) \rtimes \Gamma$ \cite{Nic19}.

Such emerging interactions between group theory and operator algebra theory are not mere coincidence. The subject of operator algebras is a branch of functional analysis sharing deep connections with other areas of pure Mathematics and Physics. The study focuses on subalgebras of operators on a complex Hilbert space, that are closed under appropriate topologies. Recent applications of operator algebraic techniques to representation theory of infinite discrete groups \cite{KK17, BKKO17} stem from studying non-commutative boundaries of non-self-adjoint operator algebras in the sense of Arveson \cite{Arv69, DK13}. These works eventually led to the resolution of open problems in group theory \cite{lB17}, exciting new techniques in stationary dynamics \cite{HK+} and progress in ergodic theory of lattices in semisimple Lie groups \cite{BH21, BBHP22}.

\medskip
One of the main motivations for our work comes from a central question in the theory of C*-algebras, of the existence of co-universal quotients for various kinds of Toeplitz C*-algebras. Co-universality is realized in various forms in C*-algebra theory, and some of its early manifestation are due to Cuntz and Krieger \cite{CK80,Cu81} in their works on C*-algebras associated to topological Markov chains. The uniqueness of such co-universal quotients is often known in the literature as \emph{uniqueness theorems}, including the so-called \emph{gauge-invariant uniqueness theorem} as prototypical example. Later developments include co-universal quotients for Toeplitz-Cuntz-Pimsner C*-algebras \cite{Pim95, Kat04}, as well as more recent works on co-universal quotients for Toeplitz algebras of product systems \cite{CLSV11, DK20, DKKLL22}. In these works, state-of-the art techniques for establishing the existence of such co-universal quotient C*-algebras rely heavily on non-commutative boundary theory.

The class of Toeplitz C*-algebras we are concerned with arise from \emph{subproduct systems}, which were initially studied in the work of Shalit and Solel \cite{SS09}. Such Toeplitz C*-algebras behave very differently from those arising from product systems. For instance, the question of co-universality does not lend itself available to non-commutative boundary techniques (see \cite[Corollary 3.16]{DOM16}), and the existence of a co-universal quotient usually requires adding additional symmetries (see \cite[Example 2.3]{Vis12}). Thus, completely new techniques are often necessary in order to prove the existence of a natural co-universal quotient in various scenarios (see for instance \cite{AK21, HN+}). By using deep results from the theory of random walks \cite{Gou14, Woe21}, the first author established the existence of a natural co-universal quotient for Toeplitz C*-algebras arising from random walks, when the random walks are symmetric, aperiodic, and on non-elementary hyperbolic groups \cite[Corollary 5.2]{DO21}. 

As part of the proof of co-universality in \cite{DO21}, a new probabilistic boundary called the \emph{ratio-limit boundary} was discovered from computing quotients of Toeplitz C*-algebras arising from random walks. The ratio-limit boundary $\partial_{\rho}\Gamma$ is defined similarly to $r$-Martin boundary, where one replaces $r$-Martin kernel $K_r$ by the ratio limit kernel $H$ given by
$$H(x,y)=\lim \frac{P^n(x,y)}{P^n(e,y)}.$$

It is not known whether these limits exist in general, and when they do we say that the random walk determined by $\mu$ has the \emph{strong ratio-limit property} (SRLP). It is known that SRLP holds for random walks on nilpotent groups \cite{Mar66}, symmetric random walks on amenable groups \cite{Ave73}, symmetric random walks on non-elementary hyperbolic groups \cite{Gou14}, and certain large classes of random walks on relatively hyperbolic groups \cite{Dus22b, DPT+}.

The key insight in \cite{DO21} was that, under the assumption of SRLP, the question of existence of a co-universal quotient is equivalent to showing the existence of a unique closed minimal $\Gamma$-invariant subspace of ratio-limit boundary $\partial_{\rho}\Gamma$. In a companion paper, Woess \cite{Woe21} studied ratio limit boundary and showed that when $\Gamma$ is hyperbolic, ratio-limit boundary coincides with Gromov boundary $\partial \Gamma$. This was used as the last missing piece for the proof for co-universality in \cite{DO21}, and gives new applications of Ancona inequalities at $r=R$ beyond probability theory.

\medskip
Many questions were left open in \cite{DO21} and \cite{Woe21} and we will address some of them here in the wider context of \emph{relatively hyperbolic groups}. Relatively hyperbolic groups are a generalization of hyperbolic groups which exhibit \emph{some} hyperbolic behavior while allowing for arbitrary subgroups. This includes arbitrary free products of finitely generated groups, as well as fundamental groups of finite volume manifolds with pinched negative scalar curvature.

The class of relatively hyperbolic groups is a rich source of examples that show how various properties depend on the underlying random walk and not on the group itself. For instance, Cartwright \cite{Car88, Car89} disproved a conjecture of Gerl by exhibiting a relatively hyperbolic group carrying two finitely supported, admissible and symmetric random walks with distinct local limit theorems.
Also, Woess \cite{Woe86} and the second and third authors \cite{DG21} proved that the homeomorphism type of the $R$-Martin boundary may depend on the underlying random walk on such groups.

Ancona's results have been extended to relatively hyperbolic groups. In particular, the third author together with Gerasimov, Potyagailo and Yang \cite{GGPY21} proved a modified version of Ancona's deviation inequalities for the Green function at $r=1$, for finitely supported and admissible random walks on relatively hyperbolic groups, and used it to show that the identity map on $\Gamma$ extends to a continuous surjection from the $1$-Martin boundary to the Bowditch boundary of the group, such that the pre-images of conical points under this surjection are singletons. For relatively hyperbolic groups with respect to \emph{virtually abelian} parabolic subgroups, the second and third authors together with Gerasimov and Potyagailo \cite{DGGP20} identified the Martin boundary precisely. They showed that the preimage of any parabolic point is a sphere of dimension $d-1$ where $d$ is the rank of the stabilizer subgroup. 

Ancona deviation inequalities were extended to the $r$-Martin boundary $r\leq R$ by the second and third authors \cite{DG21}, and this allowed them to prove that the Martin boundary covers the Bowditch boundary in this setting as well. However, for $r=R$, even when the parabolic subgroups are virtually abelian, the homeomorphism type of the Martin boundary is more delicate and depends on certain spectral properties of the random walk. More precisely, the measure $\mu$ is \emph{spectrally degenerate} along a parabolic subgroup $P$ if the spectral radius of first return time Markov chain on $P$ induced from the rescaled measure $R\mu$ has spectral radius $1$ and \emph{spectrally non-degenerate} along $P$ otherwise. Roughly, $\mu$ is spectrally degenerate along $P$ when it is ``largely weighted" on $P$ from the point of view of the random walk (see Definition \ref{defspectraldegeneracy}). We call $\mu$ spectrally non-degenerate if it is spectrally non-degenerate along every parabolic subgroup.

\subsection*{Main results}
Our goal in this paper is to study ratio limit behaviour for random walks on relatively hyperbolic groups, and extend the existence of a co-universal quotient for Toeplitz C*-algebras to spectrally non-degenerate random walks on relatively hyperbolic groups.

Our first results is in Section \ref{s:min-strong-prox} show that if $\Gamma$ is a non-elementary relatively hyperbolic group, and $\mu$ is a finitely supported, symmetric and admissible probability measure, then $\Gamma \curvearrowright \overline{\partial_{M,r}^m \Gamma}$ is strongly proximal. This is made possible by the existence of the $\Gamma$-surjection from $r$-Martin boundary onto Bowditch boundary established in \cite{GGPY21}, which allows us to lift this property to $\overline{\partial_{M,r}^m \Gamma}$ as well. Density of pre-images of conical points in $\overline{\partial_{M,r}^m \Gamma}$, as well as minimality of $\Gamma \curvearrowright \overline{\partial_{M,r}^m \Gamma}$, which were proved in \cite[Corollary 1.6]{GGPY21} for $r=1$, but their proof carries mutatis-mutandis for any $r\in [1,R]$.

Spectral non-degeneracy of the random walk will play an important role in our paper, and in Section \ref{s:spect-non-deg} we show that spectrally non-degenerate random walks are ubiquitous on relatively hyperbolic groups. In fact, in Proposition \ref{propconstructionnonspecdeg} we show that spectrally non-degenerate adapted random walks exist on arbitraty non-elementary free products and in Proposition \ref{propsmallhomogeneousdimension} we show that symmetric random walks are automatically spectrally non-degenerate whenever parabolic subgroups are virtually nilpotent of homogeneous dimension at most $4$.

For $r\in [1,R]$, $s\in \mathbb{N}$ and $x,y\in \Gamma$ we denote
$$
I^{(s)}(x,y|r) := \sum_{x_1,...,x_s \in \Gamma} G(x,x_1|r)G(x_1,x_2|r) ... G(x_s,y|r).
$$
One of the key reductions used by Woess in the hyperbolic case to identify ratio-limit boundary with Gromov boundary was to show that 
\begin{equation} \label{eq:woessq}
H(x,y) = \lim_{r\rightarrow R}\frac{I^{(1)}(x,y|r)}{I^{(1)}(e,y|r)}.
\end{equation}
In \cite[Question 7.1(a)]{Woe21} Woess asked under what general conditions does equation \eqref{eq:woessq} hold. In Section \ref{s:asymp} we shed considerable light on this question in the relatively hyperbolic context. In Corollary~\ref{c:ratio-limit-iterated-sums} we provide sufficient conditions for when $H(x,y) = \lim_{r\rightarrow R}\frac{I^{(s)}(x,y|r)}{I^{(s)}(e,y|r)}$ where $\Gamma$ is relatively hyperbolic, and $s \in \mathbb{N}$ is the smallest for which the $s$-derivative $G^{(s)}(x,y|r)$ diverges as $r\rightarrow R$ for some (all) $x,y \in \Gamma$. This leads us to Proposition~\ref{p:limit-harmonic} where necessary conditions are provided for $H(x,y) = \lim_{r\rightarrow R}\frac{I^{(s)}(x,y|r)}{I^{(s)}(e,y|r)}$ to hold, and provides us with the following example.

\vspace{6pt}

\noindent {\bf Example \ref{exampleallderivativefinite}} \emph{There exists a finitely supported, admissible, symmetric and adapted random walk on a free products for which the validity of equation \eqref{eq:woessq} fails. In fact, $H(x,y) = \lim_{r\rightarrow R}\frac{I^{(s)}(x,y|r)}{I^{(s)}(e,y|r)}$ can fail for \emph{any} $s$ for an adapted random walk on a free product.
}

\vspace{6pt}

Examples of this kind, and the necessary conditions in Proposition \ref{p:limit-harmonic}, show the limitation of the original strategy employed by Woess in \cite{Woe21}. It shows that new techniques are necessary for studying ratio-limit boundaries, even for adapted random walks on free products.

\medskip
Despite these aforementioned limitations, the strategy employed by Woess in \cite{Woe21} can still be carried out for many random walks on relatively hyperbolic groups, including spectrally non-degenerate random walks. In Section \ref{s:main} we employ results of the second author from \cite{Dus22a, Dus22b} to show that for a spectrally non-degenerate random walk on a relatively hyperbolic group $\Gamma$, for a point $\xi \in \partial_{M,R}\Gamma$ whose image in Bowditch boundary is conical, we have that $\frac{H(x,y)}{K_R(x,y)} \rightarrow 1$ as $y\rightarrow \xi$. This allows us show in Corollary \ref{cor:embedding} that closure of minimal points $\overline{\partial^m_{M,R}\Gamma}$ embeds inside ratio-limit boundary $\partial_{\rho}\Gamma$ via a bi-Lipschitz $\Gamma$-equivariant map. The main result of Section \ref{s:main} is Theorem \ref{t:essent-min} which shows that $\Gamma \curvearrowright \partial_{\rho}\Gamma$ has a unique smallest closed $\Gamma$-invariant subspace, which is $\Gamma$-homeomorphic to $\overline{\partial^m_{M,R}\Gamma}$. 

\vspace{6pt}

\noindent {\bf Theorem \ref{t:essent-min}} \emph{
Let $\Gamma$ be a non-elementary relatively hyperbolic group and let $\mu$ be a finitely supported, symmetric, admissible and aperiodic probability measure on $\Gamma$. Assume further the $\mu$-random walk is spectrally non-degenerate.
Then, there is a bi-Lipschitz $\Gamma$-equivariant map $i:\overline{\partial^m_{M,r}\Gamma}\to \partial_\rho\Gamma$.
The action on the ratio-limit boundary $\Gamma \curvearrowright \partial_{\rho} \Gamma$ is essentially minimal, i.e. it has a unique smallest closed $\Gamma$-invariant subspace which coincides with the action $\Gamma \curvearrowright \overline{\partial^m_{M,R}\Gamma}$.}

\vspace{6pt}

We expect that most of our results hold without assuming that the random walk is spectrally non-degenerate. Although the geometric arguments are really more difficult to handle, we can still prove that some of our results holds for certain convergent random walks on the free product $\mathbb{Z}^5 * \mathbb{Z}$. This allows us to show that there are two distinct random walks on $\mathbb{Z}^5 * \mathbb{Z}$ for which the associated ratio-limit boundaries are not equivariantly homeomorphic.

\vspace{6pt}

\noindent {\bf Example \ref{ex:non-iso-rl-bndry}} \emph{
There exist two finitely supported, symmetric, admissible and adapted probability measures $\mu$ and $\mu'$ on $\Gamma:=\mathbb{Z}^5 * \mathbb{Z}$ such that $\partial_{\rho}(\Gamma,\mu)$ and $\partial_{\rho}(\Gamma,\mu')$ are not equivariantly homeomorphic.}

\vspace{6pt}

Theorem \ref{t:essent-min} is then leveraged in Section \ref{s:co-univ} to answer the co-universality question for Toeplitz C*-algebras arising from spectrally non-degenerate random walks on relatively hyperbolic groups. More precisely, in Theorem \ref{t:co-universalquotient} where we show that the Toeplitz algebra $\mathcal{T}(\Gamma,\mu)$ has a a unique smallest $\Gamma\times \mathbb{T}$ equivariant quotient which coincides $C(\overline{\partial^m_{M,R} \Gamma} \times \mathbb{T}) \otimes \mathbb{K}(\ell^2(\Gamma))$.

\vspace{6pt}

\noindent {\bf Theorem \ref{t:co-universalquotient}} \emph{
The Toeplitz algebra $\mathcal{T}(\Gamma,\mu)$ has a a unique smallest $\Gamma\times \mathbb{T}$ equivariant quotient which coincides with $C(\overline{\partial^m_{M,R} \Gamma} \times \mathbb{T}) \otimes \mathbb{K}(\ell^2(\Gamma))$.}

\vspace{6pt}

This provides a positive answer to the first part of \cite[Question 5.6]{DO21} for spectrally non-degenerate random walks on relatively hyperbolic groups. As an application of co-universality applied to the random walks in Example \ref{ex:non-iso-rl-bndry}, we show in Example \ref{ex:no-gamma-t-quotient-iso} that no two $\Gamma \times \mathbb{T}$-equivariant quotients of Toeplitz C*-algebras, one of $\mathcal{T}(\Gamma,\mu)$ and one of $\mathcal{T}(\Gamma,\mu')$, are $\Gamma \times \mathbb{T}$-equivariantly isomorphic.

Our work provides first examples of a finitely generated group for which the the ratio-limit boundary, as well as the Toeplitz C*-algebra and its equivariant quotients depend on the finitely supported, symmetric and admissible random walk. We emphasize that such examples cannot be found among hyperbolic groups, because in that situation both the Martin boundary and the ratio limit boundary coincide with the Gromov boundary. This brings new impetus for the study of ratio limit boundaries on relatively hyperbolic groups.

\subsection*{Acknowledgements} The authors are grateful to Alon Dogon, Victor Gerasimov and Leonid Potyagailo for comments, remarks and suggestions on previous versions of this paper.

	
\section{Random walks and relatively hyperbolic groups.} \label{s:rw-rh}

In this section we discuss necessary material from the the theory of random walks and the theory of relatively hyperbolic groups. We refer the reader to \cite{Woe00, Woe09, Woe21} for detailed treatments of topological boundary theory for random walks, and to \cite{Bow99, Bow12, Qui+} for some of the basic theory of relatively hyperbolic groups.

\begin{Def}
Let $\Gamma$ be a discrete group, and $\mu : \Gamma \rightarrow [0,1]$ a probability measure such that $\supp(\mu)$ generates $\Gamma$ as a semigroup. We call such a probability measure \emph{admissible}. We will say that $\mu$ is \emph{finitely supported} if the set of $g\in \Gamma$ for which $\mu(g)>0$ is finite. The transition kernel $P$ on $\Gamma$ given by $P(x,y) = \mu(x^{-1}y)$ is called the \emph{random walk} on $\Gamma$ determined (or driven) by $\mu$.
\end{Def}

The iterates of $P$ are given by $P^{n}(x,y) = \mu^{*n}(x^{-1}y)$ where $\mu^{*n}$ is the $n$-th convolution power of $\mu$. We denote by $R:= R(P)$ the inverse of the spectral radius of $P$ given by
$$R(P):= \big(\limsup_{n\rightarrow \infty}\sqrt[n]{P^{n}(x,y)}\big)^{-1},$$
which is independent of the choice of $x,y\in G$. By a theorem of Kesten \cite{Kes59}, we know that $R>1$ whenever $G$ is non-amenable.

\subsection*{Martin boundaries and $t$-harmonic functions}
One way of measuring the behavior of a random walk at infinity arises from considering compactifications of $\Gamma$ with respect to Green functions. Let $\mu$ be a finitely supported and admissible probability measure on $\Gamma$. The Green function of $\mu$ for $x,y\in G$ is given by
$$
G(x,y|r) := \sum_{n=0}^{\infty} P^{n}(x,y) r^n,
$$
with radius of convergence $R$. The \emph{$r$-Martin kernel} $K_r: \Gamma \times \Gamma \rightarrow (0,\infty)$ at $r \in [1,R]$ is then given by
$$
K_r(x,y) := \frac{G(x,y|r)}{G(e,y|r)},
$$
and is well-defined even for $r= R$ by \cite[Lemma 3.66]{Woe09}. Then, for any $r\in [1, R]$, the compactification of $\Gamma$ (see for instance \cite[Theorem 7.13]{Woe09}) with respect to the $r$-Martin kernel functions $\{y \mapsto K_r(x,y)\}_{x\in \Gamma}$ defines the so-called \emph{$r$-Martin compactification} $\Delta_{M,r}\Gamma$, so that $y \mapsto K_r(x,y)$ extend to a continuous functions on $\Delta_{M,r}\Gamma$.

In fact, since for every $x\in \Gamma$ we have $C_x>0$ so that $K_r(x,y) \leq C_x$, the $r$-Martin compactification $\Delta_{M,r}\Gamma$ can also be obtained by completing the bounded metric $d_{M,r}$ on $\Gamma$ given by
\begin{equation}\label{defmetricMartin}
    d_{M,r}(y,y') = \sum_{x\in \Gamma} \frac{|K_r(x,y) - K_r(x,y')| + |1_{y} - 1_{y'}|}{2^{\phi(x)}C_x},
\end{equation}
where $1_y$ is the characteristic function of $y$, and $\phi : \Gamma \rightarrow \mathbb{N}$ is some bijection.

The action of $\Gamma$ on itself by left multiplication extends to an action by homeomorphisms $\Gamma\curvearrowright \Delta_{M,r}\Gamma$, and the \emph{$r$-Martin boundary} is then the closed $\Gamma$-invariant subspace $\partial_{M,r} \Gamma = \Delta_{M,r}\Gamma \setminus \Gamma$. A sequence $g_n \in \Gamma$ converges to a point $\xi \in \partial_{M,r} \Gamma$ if $g_n$ is outside any finite set and $\lim_n K_r(x,g_n) = K_r(x,\xi)$.

The $r$-Martin compactification is intimately related to $r^{-1}$-harmonic functions. We refer to \cite[Section 24]{Woe00} for the theory of harmonic functions, which is done for general Markov chains (and not just random walks). Most of the literature deals with the case when $r=1$. To obtain results for the case when $r\in (1, R]$, one often applies a \emph{Doob transform} or \emph{$u$-process} to $P$ via some positive $r^{-1}$-harmonic function $u$ to reduce to the case when $r=1$ (see \cite[Page 261]{Woe00}).

We say that a positive function $u : \Gamma \rightarrow (0,\infty)$ is \emph{$t$-harmonic} if we have $t \cdot u(x) = \sum_{y\in \Gamma}P(x,y)u(y)$ for every $x\in \Gamma$. We say that $u$ is \emph{normalized} if $u(e) =1$, and we denote by $\mathcal{H}^+(P,t)$ the set of all positive normalized $t$-harmonic functions for $P$, which we imbue with the topology of pointwise convergence.

The key attribute of $\partial_{M,r} \Gamma$ is that every normalized positive $r$-harmonic function $u$ on $\Gamma$ has a Poisson-Martin integral representation
\begin{align} \label{eq:martin-rep}
u(x) = \int_{\partial_{M,r} \Gamma} K_r(x,\xi) d\nu^u(\xi),
\end{align}
where $\nu^u$ is a Borel probability measure on $\partial_{M,r} \Gamma$. We call a function $u$ in $\mathcal{H}^+(P,t)$ \emph{minimal} if it cannot be written as a convex combination of two distinct functions in $\mathcal{H}^+(P,t)$. The \emph{minimal $r$-Martin boundary} $\partial_{M,r}^m \Gamma$ is the Borel set of points $\xi \in \partial_{M,r} \Gamma$ for which $x \mapsto K(x,\xi)$ is minimal, or equivalently those points $\xi$ for which $x \mapsto K(x,\xi)$ has a unique representing measure, which must hence be a Dirac measure $\delta_{\xi}$. By \cite[Corollary~24.18]{Woe00} we know that the measure $\nu^u$ in equation \eqref{eq:martin-rep} would be the unique representing measure for $u$ if we require that $\nu^u(\partial_{M,r}^m \Gamma) = 1$. 

When $\mu$ is finitely supported and admissible, by \cite[Lemma 24.16]{Woe00} the positive normalized functions $x \mapsto K_r(x,\xi)$ are $r^{-1}$-harmonic for $\xi \in \partial_{M,r}\Gamma$, so that whenever $\nu \in \Prob(\partial_{M,r}\Gamma)$, the function 
$$
u_{\nu}(x) = \int_{\partial_{M,r} \Gamma} K_r(x,\xi) d\nu(\xi)
$$
is in $\mathcal{H}^+(P,r^{-1})$. Let us imbue $\Prob(\partial_{M,r}\Gamma)$ with the weak* topology, which is the smallest topology such that for every continuous function $f:\partial_{M,r}\Gamma \to \mathbb R$ the map $\nu\mapsto \int f(\xi)d\nu(\xi)$ is continuous. Note that this is often called the weak topology on $\Prob(\partial_{M,r}\Gamma)$ in literature on probability. Define a map $\varphi : \Prob(\partial_{M,r} \Gamma) \rightarrow \mathcal{H}^+(P,r^{-1})$ given by $\varphi(\nu) = u_{\nu}$.
Since for every $x\in \Gamma$ the function $K_r(x,\xi)$ is continuous on $\partial_{M,r} \Gamma$, it is easy to see that $\varphi$ is continuous. That is, if $\nu_{\alpha} \rightarrow \nu$ in the weak* topology, then $u_{\nu_{\alpha}}$ converges to $u_{\nu}$ pointwise. Since every $u \in \mathcal{H}^+(P,r^{-1})$ has some representing measure $\nu^u$, we see that $\varphi$ is a continuous surjective map.

\subsection*{Ratio-limit boundaries}

Suppose now that the random walk defined by $\mu$ on $\Gamma$ is aperiodic in the sense that there exists $n_0$ such that $\mu^{*n}(e) > 0$ for any $n\geq n_0$, and let $\omega$ be a non-principle ultrafilter on $\mathbb{N}$. By analogy with $R$-Martin kernel functions, for $x,y\in \Gamma$ we may define the \emph{ratio-limit kernel} $H: \Gamma\times \Gamma \rightarrow (0,\infty)$ by
\begin{equation} \label{eq:rl}
H(x,y) := \lim_{m \rightarrow \omega} \frac{P^{m}(x,y)}{P^{m}(e,y)}.
\end{equation}

By a result of Gerl \cite{Ger73, Ger78}, we know that $\lim_{m\rightarrow \infty}\frac{P^{m+1}(x,z)}{P^m(x,z)} = R^{-1}$, so that the function $x\mapsto H(x,y)$ is $R^{-1}$-harmonic for every $y\in \Gamma$.

The compactification of $\Gamma$ with respect to the ratio-limit kernel functions $\{y \mapsto H(x,y)\}_{x\in \Gamma}$ defines the ratio-limit compactification $\Delta_{\rho} \Gamma$. Just like Martin compactifications, so too can the ratio-limit compactification $\Delta_{\rho} \Gamma$ be obtained by completing a bounded metric. Indeed, for every $x\in \Gamma$ we have $D_x>0$ so that $H(x,y) \leq D_x$, and the compactification $\Delta_{\rho} \Gamma$ is obtained by completing the metric $d_{\rho}$ on $\Gamma$ given by
\begin{equation}\label{defmetricratio}
d_{\rho}(y,y') = \sum_{x\in \Gamma} \frac{|H(x,y) - H(x,y')| + |1_{y} - 1_{y'}|}{2^{\phi(x)}D_x},
\end{equation}
where $1_y$ is the characteristic function of $y$, and $\phi : \Gamma \rightarrow \mathbb{N}$ is some bijection.

Again the left multiplication of $\Gamma$ on itself extends to an action $\Gamma\curvearrowright \Delta_{\rho} \Gamma$ by homeomorphisms, and the closed $\Gamma$-invariant subspace $\partial_{\rho} \Gamma := \Delta_{\rho} \Gamma \setminus \Gamma$ is called the \emph{ratio-limit boundary}. A sequence $g_n \in \Gamma$ converges to a point $\xi \in \partial_{\rho} \Gamma$ if $g_n$ is outside any finite set and $\lim_n H(x,g_n) = H(x,\xi)$. 

We will say that the random walk determined by $\mu$ has the \emph{strong ratio-limit property} (SRLP) if the following limits exist for every $x,y\in \Gamma$.
\begin{equation*}
H(x,y) := \lim_{m \rightarrow \infty} \frac{P^{m}(x,y)}{P^{m}(e,y)}.
\end{equation*}

\begin{Rmk}
Note that when the random walk determined by $\mu$ has SRLP, all of the above defined notions are independent of the non-principle ultrafilter $\omega$. It is still unknown to us whether the limits always exist as $m \rightarrow \infty$, and whether the definition of ratio-limit boundary is independent of the choice of non-principle ultrafilter.
\end{Rmk}

Since $x\mapsto H(x,y)$ is $R^{-1}$-harmonic for every $y\in \Gamma$ and the random walk we consider is finitely supported and  admissible, it follows that $H(x,\eta)$ is $R^{-1}$-harmonic for every $\eta \in \partial_{\rho}\Gamma$. Hence, by Martin--Poisson representation theorem, for every $\eta \in \Delta_{\rho} \Gamma$ there is a probability measure $\nu^{\eta} \in \Prob(\partial_{M,R}\Gamma)$ such that
$$
H(x,\eta) = \int_{\partial_{M,R}\Gamma} K_R(x,\xi)d\nu^{\eta},
$$
which is unique if we require that $\nu^{\eta}(\partial_{M,R}^m \Gamma) = 1$.

Since distinct elements of $\Delta_{\rho} \Gamma$ may give rise to the same ratio-limit kernels, it makes sense to discuss a ``separated" version of the ratio-limit compactification and boundary.

\begin{Def} \label{d:ratio-limit-radical}
Suppose $\Gamma$ is a discrete group, and $\mu$ an aperiodic probability measure on $\Gamma$. Let $\omega$ be a non-principle ultrafilter on $\mathbb{N}$. We define the ratio-limit radical as the subset
$$
R_{\mu}^{\omega} := R_{\mu}:= \{ \ g\in \Gamma \ | \ H(x,g) = H(x,e), \ \forall x \in \Gamma \ \}.
$$
\end{Def}

In other words, $R_{\mu}$ is the largest subset of elements $g \in \Gamma$ of all $R^{-1}$-harmonic functions $x\mapsto H(x,g)$ that coincide with the $R^{-1}$-harmonic function $x\mapsto H(x,e)$. When $\mu$ is symmetric, $g\in R_{\mu}$ if and only if for every $y \in \Gamma$ we have $H(g,y) = H(g,e)$. In this case, the next result improves upon \cite[Proposition 3.2]{DO21}, showing that $R_{\mu}$ is always a \emph{normal} subgroup.

\begin{Prop}
Let $\Gamma$ be discrete, $\mu$ an aperiodic probability measure on $\Gamma$, and $\omega$ is an non-principle ultrafilter. Then $R_{\mu}$ is a subgroup. If moreover $\mu$ is symmetric, then $R_{\mu}$ is normal.
\end{Prop}

\begin{proof}
First note that we have the cocycle identity $H(gh,z) = H(h,g^{-1}z)H(g,z)$ for $h,g,z\in \Gamma$.
Hence, if $g,h\in R_{\mu}$ we get that
$$
H(x,gh) = \frac{H(g^{-1}x,h)}{H(g^{-1},h)} = \frac{H(g^{-1}x,e)}{H(g^{-1},e)} = H(x,g) = H(x,e)
$$
and
$$
H(x,g^{-1}) = \frac{H(gx,e)}{H(g,e)} = \frac{H(gx,g)}{H(g,g)} = H(x,g^{-1}g) = H(x,e).
$$
Thus, $R_{\mu}$ is a subgroup. Next, if $\mu$ is symmetric, we have another cocycle identity given by $H(gh,z) = H(g,zh^{-1})H(h,z)$ for $h,g,z\in \Gamma$. Hence, if $g \in R_{\mu}$ and $h\in \Gamma$, we also have
$$
H(x,h^{-1}gh) = \frac{H(hx,gh)}{H(h,gh)} = \frac{H(hxh^{-1},g) / H(h^{-1},g)}{H(hh^{-1},g) / H(h^{-1},g)} = H(hxh^{-1},g) =
$$
$$
H(hxh^{-1},e) = H(xh^{-1},h^{-1})H(h,e) = H(x,e)H(h^{-1},h^{-1})H(h,e).
$$
But now, using homogeneity we get that $H(h^{-1},h^{-1})H(h,e) = 1$, and we obtain $H(x,h^{-1}gh) = H(x,e)$. Hence, $h^{-1}gh \in R_{\mu}$, and we see that $R_{\mu}$ is normal.
\end{proof}

Next, note that the function $(x,yR_{\mu}) \mapsto H(x,y)$ is well-defined on $\Gamma \times \Gamma/R_{\mu}$, and we may use them to define the \emph{reduced ratio-limit compactification} $\Delta_{\rho}^r \Gamma$. This is obtained by completing the space $\Gamma/R_{\mu}$ with respect to the metric $d_{\rho}^r$ on $\Gamma/R_{\mu}$ given by
$$
d_{\rho}^r(yR_{\mu},y'R_{\mu}) = \sum_{x\in \Gamma} \frac{|H(x,y) - H(x,y')|}{2^{\phi(x)}D_x},
$$
where $\phi : \Gamma \rightarrow \mathbb{N}$ is some bijection. Left multiplication of $\Gamma$ on $\Gamma / R_{\mu}$ then extends to an action $\Gamma \curvearrowright \Delta_{\rho}^r\Gamma$ by homeomorphisms, and the closed $\Gamma$-invariant subspace given by $\partial_{\rho}^r \Gamma : = \Delta_{\rho}^r \Gamma \setminus [\Gamma/R_{\mu}]$ is called the \emph{reduced ratio-limit boundary}.

By \cite[Lemma 6.3]{Woe21} there is a $\Gamma$-equivariant continuous surjection $\Delta_{\rho} \Gamma \rightarrow \Delta_{\rho}^r \Gamma$ which is injective on $\partial_{\rho} \Gamma$. Thus, $R^{-1}$-harmonic functions arising from compactifying via ratio-limit kernels, are parameterized by elements of $\Delta_{\rho}^r \Gamma$. Since in this paper we will be mostly concerned with non-amenable groups, the following corollary shows that we may always assume that $\partial_{\rho}^r \Gamma$ is non-empty. 

\begin{Cor}
Let $\Gamma$ be discrete, $\mu$ a \emph{symmetric} aperiodic probability measure on $\Gamma$, and $\omega$ a non-principle ultrafilter. Then $\partial^r_{\rho} \Gamma = \emptyset$ if and only if $\Gamma$ is amenable.
\end{Cor}

\begin{proof}
Suppose that $\partial^r_{\rho} \Gamma = \emptyset$. Then we must have that $\Gamma/R_{\mu}$ is finite. Since for each $g\in R_{\mu}$ we have $H(g,g) = H(g,e)$, and as $\mu$ is symmetric, we get that
$$
\lim_n \frac{P^n(e,g)}{P^n(e,e)} = \lim_n \frac{P^n(e,e)}{P^n(e,g)},
$$
which must hence be equal to $1$. The set of elements $g\in \Gamma$ satisfying $\lim_n \frac{P^n(e,g)}{P^n(e,e)} = 1$ was shown to be an amenable subgroup by \cite[Theorem~4.1 \& Theorem~4.2]{ER19}. Thus, we deduce that $R_{\mu}$ is amenable. Since $R_{\mu}$ is amenable and $\Gamma/R_{\mu}$ is finite, we get that $\Gamma$ is amenable.

Conversely, if $\Gamma$ is amenable and $\mu$ is symmetric, a result of Avez \cite{Ave73} shows that $H(x,y)=1$ for all $x,y\in \Gamma$. Thus, $\Gamma= R_{\mu}$, and hence $\partial_{\rho}^r \Gamma = \emptyset$.
\end{proof}

\subsection*{Relatively hyperbolic group and their intrinsic geometry}
Let $s$ be an isometry on a hyperbolic space $X$. We will say that $s$ is hyperbolic if it has two fixed points $x^\pm_s$ on $\partial X$. Such an isometry then has North-South dynamics on $\partial X$, in the sense that $s^{\pm n} x\underset{n\rightarrow \infty}{\to} x^\pm_s$ for any $x\in \partial X \setminus x^\mp_s$.

Suppose now that $\Gamma$ is a non-elementary, finitely generated group of isometries of some proper, geodesic hyperbolic space $X$. We denote by $\Lambda_\Gamma$ the closed set of limit points of $\Gamma x$ in the Gromov boundary $\partial X$ for some element $x\in X$ (whose definition does not depend on $x$), or equivalently the unique minimal closed $\Gamma$ invariant subset of $\partial X$.   Let $s \in \Gamma$ be hyperbolic as an isometry on $X$. It is well-known that the set of all attractors and repellers of hyperbolic elements
$$\{ \ x_s^{\pm} \ | \ s \in \Gamma, \ \text{hyperbolic} \ \}$$
is dense in $\Lambda_\Gamma$, and that $\Gamma \curvearrowright \Lambda_\Gamma$ is minimal (see for instance \cite[Proposition 4.7]{Qui+}).

Let $\Gamma$ be a finitely generated discrete group and $\Omega$ a finite collection of subgroups. We construct the coned-off (sometimes called "relative")  Cayley graph $\Gr(\Gamma;\Omega)$ as follows. Consider the usual Cayley graph $\Gr(\Gamma)$, and for each $P\in \Omega$ and coset $gP$ we add a vertex $gP$ and an edge of length $\frac{1}{2}$ from $h$ to $gP$ for any $h\in gP$. The resulting graph is $\Gr(\Gamma;\Omega)$, and the shortest path metric makes $\Gr(\Gamma;\Omega)$ into a geodesic metric space. We will call the shortest path metric on $\Gr(\Gamma)$ the \emph{word metric} and the shortest path metric on $\Gr(\Gamma;\Omega)$ the \emph{relative metric}. A geodesic in $\Gr(\Gamma)$ will be called a \emph{word geodesic}, and a geodesic in $\Gr(\Gamma,\Omega)$ will be called a \emph{relative geodesic}.

\begin{Def}
Let $\Gamma$ be a finitely generated discrete group and $\Omega$ a finite collection of subgroups. We say that $\Gamma$ is \emph{hyperbolic relative to $\Omega$} if
\begin{enumerate}
\item $\Gr(\Gamma;\Omega)$ is $\delta$ hyperbolic for some $\delta>0$

\item For each $L \in \mathbb{N}$, any edge $e$ in $\Gr(\Gamma;\Omega)$ belongs to at most finitely many simple cycles of length $L$.

\end{enumerate}
\end{Def}

We let $\partial \Gr(\Gamma,\Omega)$ be the Gromov boundary of $\Gr(\Gamma,\Omega)$ and we let $V_\infty(\Gr(\Gamma,\Omega))$ be the set of vertices with infinite degree in $\Gr(\Gamma,\Omega)$.
We also set $\partial_B(\Gamma;\Omega)=\partial \Gr(\Gamma,\Omega)\cup V_\infty(\Gr(\Gamma,\Omega))$.
In \cite[Section~8]{Bow12}, Bowditch defines a topology on $\Gamma\cup \partial_B(\Gamma;\Omega)$ such that the induced topology on $\Gr(\Gamma,\Omega)\cup \partial\Gr(\Gamma,\Omega)$ is the visual topology on the Gromov completion of a hyperbolic space  \cite[Proposition~8.5]{Bow12} and such that $V_\infty(\Gr(\Gamma,\Omega))$ is dense in $\partial_B(\Gamma;\Omega)$.
The space $\partial_B(\Gamma;\Omega)$ is compact \cite[Proposition~8.6]{Bow12} and is called the Bowditch boundary of $\Gamma$.

The action of $\Gamma$ on $\partial_B(\Gamma;\Omega)$ is minimal and geometrically finite. That is, it is a convergence action (i.e. the action on the space of discrete triples is properly discontinuous) and every point in $\partial_B(\Gamma;\Omega)$ is either conical or bounded parabolic, where $\xi \in \partial_B(\Gamma;\Omega)$ is

\begin{enumerate}
\item  a \emph{conical point} if there is a sequence $(g_n)$ in $\Gamma$ and two distinct points $\xi_1,\xi_2 \in \partial_B(\Gamma;\Omega)$ such that for any $\xi \neq \zeta \in \partial_B(\Gamma;\Omega)$, the sequences $(g_n\xi)$ and $(g_n \zeta)$ converge to $\xi_1$ and $\xi_2$ respectively, or;

\item a \emph{bounded parabolic point} if the stabilizer $\Gamma_{\xi}$ is infinite and acts cocompactly on $\partial_B(\Gamma;\Omega)\setminus \{\xi\}$.
\end{enumerate} 

Relative hyperbolicity and Bowditch boundary can also be identified abstractly as follows. Equivalently, we say that $\Gamma$ is relatively hyperbolic if there is some proper, geodesic hyperbolic space $X$ on which $\Gamma$ acts by isometries on $X$, geometrically finitely on $\partial X$, and the maximal parabolic subgroups (i.e. infinite stabilizers of points of $\partial X) $ are conjugate to elements of $\Omega$. An infinite order element $g$ in a relatively hyperbolic group $\Gamma$ is then called hyperbolic if it is a hyperbolic isometry on a space $X$ as above. The limit set $\Lambda_\Gamma$ in the Gromov boundary of such a space is then unique up to $\Gamma$-equivariant homeomorphism \cite[Theorem~9.4]{Bow12}, and coincides with the Bowditch boundary $\partial_B(\Gamma;\Omega)$ \cite[Proposition~9.1]{Bow12}.
Under this identification, $\partial\Gr(\Gamma,\Omega)$ is the set of conical limit points and $V_\infty(\Gr(\Gamma,\Omega))$ is the set of bounded parabolic limit points.
It follows from this definition that $x_s^{\pm}$ are conical for every hyperbolic element $s\in \Gamma$, and that there are only countably many bounded parabolic points in $\partial_B(\Gamma;\Omega)$.

Let $\alpha$ be a word geodesic of a relatively hyperbolic group $\Gamma$. A point $p\in \alpha$ is said to be an $(\epsilon_1,\epsilon_2)$ transition point if the length $\epsilon_2$ sub-segment of $\alpha$ centered at $p$ is not contained in an $\epsilon_1$ neighborhood of a single coset of a parabolic subgroup.
When $\Gamma$ is relatively hyperbolic, by \cite[Proposition~8.13]{Hru10} we know that relative geodesics are uniformly close to geodesics in $\Gr(\Gamma)$ with respect to the Hausdorff distance. More precisely, there exists $C>0$ such that the following holds for any two $x,y\in \Gamma$.
Inside $\Gr(\Gamma)$, the set of points in $\Gamma$ on a relative geodesic from $x$ to $y$ is within Hausdorff distance $C$ of the set of transition points on a geodesic from $x$ to $y$.

By \cite[Lemma 2.20]{Yang} we know there exist $\epsilon_1, \epsilon_2 >0$ such that for every conical $x \in \partial_B (\Gamma;\Omega)$ and any word geodesic ray $\gamma$ converging to $x$ there is a sequence of $(\epsilon_1,\epsilon_2)$ transition points $p_n\in \gamma$ converging to $x$. Moreover, by \cite[Lemma 2.4]{DGdrift}  (which itself is a simple corollary of \cite[Lemma 1]{Karlsson}) there exists $D>0$ such that for any triangle whose sides are word geodesic segments or rays with vertices in $\Gamma\cup \partial_B (\Gamma;\Omega)$, any $(\epsilon_1,\epsilon_2)$ transition point on one side is within $D$ of an $(\epsilon_1,\epsilon_2)$ transition point on one of the other two. We will need the following lemma for what is to follow.

\begin{Lemma} \label{l:tpointsr}
There are $\epsilon_1,\epsilon_2>0$ such that for all $g\in \Gamma$ there is a conical $x \in \partial_B \Gamma$ such that some word geodesic $[e,x]$ contains an $(\epsilon_1,\epsilon_2)$ transition point within $2\epsilon_2$ of $g$. 
\end{Lemma}

\begin{proof}
Let $\Omega_{D,\epsilon_1,\epsilon_2}(g)$ denote the set of $x \in \partial_B (\Gamma;\Omega)$ such that word some geodesic $[e,x]$ contains an $(\epsilon_1,\epsilon_2)$ transition point within $2\epsilon_2$ of $g$ and some point within $D$ of $g$. By \cite[Lemma 5.2]{Yang}, there exist constants $D',\epsilon_1,\epsilon_2>0$ for which $\Omega_{D',\epsilon_1,\epsilon_2}(g)$ has positive Patterson-Sullivan measure. Moreover by \cite[Theorem 1.7]{Yang} the Patterson-Sullivan has no atoms, and must hence give full weight to conical points. Hence, the intersection of $\Omega_{D',\epsilon_1,\epsilon_2}(g)$ with the set of conical points still has positive Patterson-Sullivan measure, and the result follows.
\end{proof}

\subsection*{Deviation inequalities of Green functions along geodesics}

The following versions of Ancona inequalities for relative hyperbolic groups were first proved in \cite[Corollary 9.2]{GGPY21} when $r=1$ and then in \cite[Theorem 3.6]{DG21} uniformly for $r\leq R$. We will call them the \emph{weak relative Ancona inequalities} to contrast with a stronger version which will be described below. We remark that the term "relative" is used to insist on the fact that in relatively hyperbolic groups, these deviation inequalities do not hold along word geodesics, but only along transition points on them, or equivalently along relative geodesics.
\begin{Prop}\label{weakrelativeAncona} 
Let $\Gamma$ be a relatively hyperbolic group and $\mu$ a finitely supported, admissible and symmetric probability measure on $\Gamma$. Then for any $\epsilon_1,\epsilon_2,D>0$ there is a $C>0$ such that for any $1\leq r\leq R$ and any $x,y,z\in \Gamma$ with $y$ within $D$ of an $(\epsilon_1,\epsilon_2)$ transition point on a word geodesic $[x,z]$ we have 
\begin{equation}\label{Ancona}
G(x,z |r)\leq CG(x,y|r)G(y,z|r)
\end{equation}
Equivalently, by \cite[Proposition~8.13]{Hru10}, for any $D>0$ there is a $C>0$ such that (\ref{Ancona}) holds for any $x,y,z$ with $y$ within word distance $D$ of an \emph{relative} geodesic from $x$ to $z$.
\end{Prop}

We will also need the following deviation inequality for random walks on relatively hyperbolic groups, which we call the \emph{strong relative Ancona inequality}. These were proved by the second and third named authors \cite[Theorem~3.14]{DG21} and \cite[Theorem~2.15]{Dus22b}.
For hyperbolic groups, these were first proved for by Izumi-Neshveyev-Okayasu \cite{INO08} for $r=1$ and by Gouez\"el \cite{Gou14} uniformly for $r\leq R$.

We say that two relative geodesics $[x,y]$ and $[x',y']$ $c$-fellow travel for a time at least $n$ if the $n$ first points of both these relative geodesics are within $c$ of each other in the word metric on the Cayley graph $\Gr(\Gamma)$.

\begin{Prop}\label{strongrelativeAncona}
    Let $\Gamma$ be a relatively hyperbolic group and $\mu$ a finitely supported, admissible and symmetric probability measure on $\Gamma$. Then for any $c>0$ there is a $C>0$ and $0<\alpha<1$ such that if $x,x',y,y'$ are four points such that relative geodesics $[x, y]$ and $[x',y']$
$c$-fellow travel for a time at least $n$ we have 
\begin{equation} \label{strongAncona}
    \left |\frac{G(x,y|r)G(x',y'|r)}{G(x,y|r)G(x',y' |r)}-1\right |\leq C\alpha^n
\end{equation}
\end{Prop}

Finally, a pivotal connection between Bowditch boundary and $r$-Martin boundary is given as follows. By \cite[Theorem~1.5, Corollary~1.7 \& Corollary~7.10]{GGPY21} we know that for any $1\leq r \leq R$ there is a $\Gamma$-equivariant surjection $\pi: \Delta_{M,r} \Gamma \rightarrow \Gamma \cup \partial_B(\Gamma;\Omega)$ such that $\pi|_\Gamma = \id_\Gamma$ and such that for any point $x \in \partial_B(\Gamma;\Omega)$ the pre-image $\pi^{-1}(x)$ contains a minimal point in $\partial_{M,r}\Gamma$, and if $x\in \partial_B(\Gamma;\Omega)$ is conical, then the pre-image $\pi^{-1}(x)$ is a singleton. Although these results are formally stated for the case of $r=1$, the proofs go through verbatim for $r$-Martin kernels with any $r \in [1,R]$. This is because weak relative Ancona inequalities hold for $r=R$ in \cite[Theorem 3.6]{DG21} when $\mu$ is symmetric, and the maximum principle for Martin kernels also holds for $r$-Martin kernels for any $r\in [1,R]$.

\section{Minimal points, minimal actions, and strong proximality} \label{s:min-strong-prox}

In this section we show that for $\Gamma$ relatively hyperbolic, and $\mu$ a finitely supported, admissible and symmetric measure on $\Gamma$, the closure of minimal points $\overline{\partial_{M,R}^{m}\Gamma}$ is a minimal and strongly proximal subspace of $R$-Martin boundary. It is a deep open problem to determine whether minimal points of $\partial_{M,R}\Gamma$ are dense when $\Gamma$ is relatively hyperbolic, and we provide examples for when we automatically have $\overline{\partial_{M,R}^{m}\Gamma} = \partial_{M,R}\Gamma$. We emphasize here that aside from strong proximality, the analogous statements for $r=1$ were proved in \cite[Corollary 1.6]{GGPY21} and the arguments carry over to the case where $r\in (1,R]$ (using the Ancona inequalities for $r=R$ from \cite[Theorem 3.6]{DG21}). We provide a self-contained proof of them here for the benefit of the reader.

In what follows, for two functions $u,w$ on $\Gamma$ we will denote $u \asymp_A w $ if for the constant $A>0$ we have $\frac{1}{A} \leq \frac{u(g)}{w(g)} \leq A$ for all $g\in \Gamma$. Suppose that $\mu$ is as above, and recall that $d_{M,R}$ is the Green metric from Section \ref{s:rw-rh}. Then, since $P$ is $\Gamma$ invariant and irreducible, we get a local Harnack inequality \cite[Equation (4.3)]{Woe21} (See also \cite[(25.1)]{Woe00}). That is, for every $g\in \Gamma$ there is a constant $A> 0$, depending only on the word distance between $g$ and $h$, such that $G(g,\cdot |r) = G(\cdot ,g |r) \asymp_A G(\cdot,h |r) = G(h,\cdot |r)$ for any $r\in [1,R]$.

\begin{Prop} (c.f. \cite[Proposition 7.12]{GGPY21})
\label{p:conicalroughlyapproximateminimal}
Let $\Gamma$ be a non-elementary hyperbolic relative to a finite collection of subgroups $\Omega$. Let $\mu$ be a finitely supported, admissible and symmetric probability measure on $\Gamma$, and take $r\in [1,R]$. Then, there exists a constant $C>0$ such that for any $\xi \in \partial_{M,r} \Gamma$ there is a sequence of conical points $\zeta_n$ in the Bowditch boundary such that the sequence $(\pi^{-1}(\zeta_n))$ converges to $\alpha \in \partial_{M,r} \Gamma$ such that for all $g\in \Gamma$ we have
$$
C\geq \frac{K_{r}(g,\alpha)}{K_{r}(g,\xi)} \geq \frac{1}{C}
$$
\end{Prop}

\begin{proof}
Let $\xi \in \partial_{M,r} \Gamma$ and $x=\pi(\xi)$. Let $(g_n)$ be a sequence in $\Gamma$ which converges to $\xi$. By Lemma~\ref{l:tpointsr} there exist $\epsilon_1,\epsilon_2 > 0$ and a sequence $(\zeta_n)$ of conical points of the Bowditch boundary for which some geodesic ray $[e,\zeta_n]$ contains an $(\epsilon_1,\epsilon_2)$ transition point within $2\epsilon_2$ of $g_n$. Note that since $\zeta_n$ are conical, any geodesic ray with endpoint $\zeta_n$ also converges to  $\pi^{-1}(\zeta_n)$ inside the $R$-Martin boundary. 

Now, by \cite[Lemma 2.4]{DGdrift} there is a constant $D>0$ such that for any $g\in \Gamma$ and large enough $n$, any geodesic $[g,\zeta_n]$ has an $(\epsilon_1,\epsilon_2)$ transition point within $D$ of $g_n$. By the relative Ancona inequality (in the form of \cite[Corollary 9.2]{GGPY21}) there exists $B>0$ independent of $n$ such that for a transition point $q_n\in [e,\zeta_n]$ within $D$ of $g_n$ we have $G(e,g_n |r)\asymp_B G(e,q_n |r)G(q_n,g_n |r)$ for a transition point $p_n\in [g,\zeta_n]$ within $D$ of $g_n$ we have $G(g,g_n |r)\asymp_B G(g,p_n |r)G(p_n,g_n |r)$. Let $p_n\in [g,\zeta_n],q_n\in [e,\zeta_n]$ be unbounded sequences of $(\epsilon_1,\epsilon_2)$ transition points with word distance (between $p_n$ and $q_n$) at most $D$ and
$$
K_r(e,q_n)/K_r(e,\pi^{-1}(\zeta_n))\in [(1+1/n)^{-1},1+1/n], 
$$
$$
K_r(g,p_n)/K_r(g,\pi^{-1}(\zeta_n))\in [(1+1/n)^{-1},1+1/n].
$$
By Harnack inequalities there is a constant $A>0$ such that $G(p_n,\cdot |r) = G(\cdot ,p_n |r) \asymp_A G(\cdot ,q_n|r) = G(q_n,\cdot|r)$, so we get that
$$
K_r(g,p_n)=\frac{G(g,p_n|r)}{G(e,p_n|r)}\asymp_A \frac{G(g,p_n|r)}{G(e,q_n|r)}  
$$
$$
\asymp_{B^2} \frac{G(g,g_n|r)G(q_n,g_n|r)}{G(e,g_n|r)G(p_n,g_n|r)} \asymp_A K_r(g,g_n).
$$ 
Hence, we get
$K_r(g, \pi^{-1}(\zeta_n))\asymp_{A^2B^2(1+\frac{1}{n})} K_r(g,g_n)\to K_r(g,\xi)$. Thus, if we take $C=A^2B^2$, any limit point $\alpha$ of the sequence $(\pi^{-1}(\zeta_n))$ will satisfy
\begin{equation*}
    C\geq K_r(g,\alpha) / K_r(g,\xi) \geq \frac{1}{C}. \qedhere
\end{equation*}
\end{proof}

\begin{Cor} \label{c:min-dense-con-dense}
Let $\Gamma$ be a non-elementary hyperbolic relative to a finite collection of subgroups $\Omega$. Let $\mu$ be a finitely supported, admissible and symmetric probability measure on $\Gamma$, and $r\in [1,R]$. Then any minimal point in the $r$-Martin boundary is a limit of a sequence of preimages of conical points. Hence, conical points are dense in $\overline{\partial_{M,r}^m\Gamma}$.
\end{Cor}

\begin{proof}
Suppose $\xi \in \partial_{M,r} \Gamma$ is minimal, and let $\alpha$ be a limit point of pre-images of conical points $\pi^{-1}(\zeta_n)$ such that $0\leq K_r(g,\alpha) \leq C K_r(g,\xi)$. Minimality of $\xi$ implies that $K_r(g,\alpha)=K_r(g,\xi)$, since these functions are normalized at $g=e$. Thus, we get that $\alpha = \xi$.
\end{proof}

The following proposition shows that the action $\Gamma \curvearrowright \overline{\partial_{M,r}^m\Gamma}$ is strongly proximal. The same result holds for $\partial_{M,r}\Gamma$ when its minimal points are dense. Recall that $\Gamma \curvearrowright X$ with $X$ compact Hausdorff is \emph{strongly proximal} if whenever $\nu \in \Prob(X)$ then the closure $\overline{\Gamma \nu}$ of its $\Gamma$ orbit contains a Dirac mass $\delta_x$ for $x\in X$. By essentially the same proof as in \cite[Corollary 1.6]{GGPY21} we know that $\Gamma\curvearrowright \overline{\partial_{M,r}^m\Gamma}$ is minimal for any $r\in [1,R]$, and we provide our own proof of this for the benefit of the reader.

\begin{Prop} \label{p:min-strong-prox}
Let $\Gamma$ be a non-elementary hyperbolic relative to a finite collection of subgroups $\Omega$. Let $\mu$ be a finitely supported, admissible and symmetric probability measure on $\Gamma$, and $r\in [1,R]$. Then the action $\Gamma\curvearrowright \overline{\partial_{M,r}^m\Gamma}$ is minimal and strongly proximal.
\end{Prop} 

\begin{proof}
Let $s\in \Gamma$ be hyperbolic for a geometrically finite action $\Gamma\curvearrowright X$. Since $x_s^{\pm}$ are conical, there pre-images are minimal points in the Martin boundary and there are unique $s$-fixed points $\xi_s^{\pm}$ in $\overline{\partial_{M,r}^m\Gamma}$ such that $\pi(\xi_s^{\pm}) = x_s^{\pm}$.
Suppose that $\xi \in \overline{\partial_{M,r}^m\Gamma} \setminus \{\xi_s^{-}\}$. We show that $s^n \xi \rightarrow \xi_s^{+}$. Indeed, by contradiction assume that $s^n\xi$ does not converge to $\xi_+$.
Then, by compactness there is a subequence of $(s^n\xi)$ which converges to $\xi'\neq \xi_+$. 
Since $\pi$ is $\Gamma$-equivariant, we get that $s^n\pi(\xi)$ converges to $x^+ = \pi(\xi^+) = \pi(\xi')$, contradicting the fact that the pre-image of the conical point $x^+$ is a singleton.

Thus, every hyperbolic element $s\in \Gamma$ has an attractor and repeller $\xi_s^{\pm}$ in $\overline{\partial_{M,r}^m\Gamma}$.
The set of $\xi \in \overline{\partial_{M,r}^m\Gamma}$, with $\pi(\xi)$ conical, is dense in $\overline{\partial_{M,r}^m\Gamma}$, and the set $\{ \ x_s^{\pm} \ | \ s\in \Gamma, \ \text{hyperbolic} \ \}$ is dense inside conical points of $\partial_B(\Gamma;\Omega)$.
Consequently, we get that the set of attractors and repellers $\{ \ \xi_s^{\pm} \ | \ s \in \Gamma, \ \text{hyperbolic} \ \}$ is dense in $\overline{\partial_{M,r}^m\Gamma}$.

Now, as $\Gamma$ is non-elementary, there are two hyperbolic elements $s,t \in \Gamma$ such that the set $\{\xi_s^{\pm}, \xi_t^{\pm}\}$ has at least three elements. We assume without loss of generality that $\xi_s^{\pm} \neq \xi_t^{-}$. We will first show that $\overline{\partial_{M,r}^m\Gamma}$ is strongly proximal.
Fix $\tau \in \Prob(\overline{\partial_{M,r}^m\Gamma})$. Then, since $s\xi_s^{-} = \xi_s^{-}$, by the dominated convergence theorem we get that
$$
s^n \tau \rightarrow \tau(\{\xi_s^{-}\}) \delta_{\xi_s^{-}} + (1-\tau(\{\xi_s^{-}\}))\delta_{\xi_s^{+}}.
$$
Also, since $\xi_s^{\pm} \neq \xi_t^{-}$ the same manipulation shows that  $t^n \xi_s^{-} \rightarrow \xi_t^{+}$, so that setting $\nu = \tau(\{\xi_s^{-}\}) \delta_{\xi_s^{-}} + (1-\tau(\{\xi_s^{-}\}))\delta_{\xi_s^{+}}$ we get that $t^m \nu \rightarrow \delta_{\xi_t^+}$. 
Thus, the closure of the $\Gamma$-orbit of $\tau$ contains a Dirac measure, so that $\Gamma \curvearrowright \overline{\partial_{M,r}^m\Gamma}$ is strongly proximal. 

Next, we show that $\Gamma \curvearrowright \overline{\partial_{M,r}^m\Gamma}$ is minimal. If we restrict the above argument to Dirac measures, we have shown above that for any element $\xi \in \overline{\partial_{M,r}^m\Gamma}$ and hyperbolic element $s\in \Gamma$, the closure of the orbit $\Gamma \xi$ contains $\xi_s^{\pm}$. Since $\overline{\partial_{M,r}^m\Gamma}$ is the closure of $\{ \ \xi_s^{\pm} \ | \ s \in \Gamma, \ \text{hyperbolic} \ \}$, we are done.
\end{proof}

There are many natural examples where minimal points are automatically dense in $\partial_{M,r}\Gamma$ , so that the action $\Gamma \curvearrowright \partial_{M,r} \Gamma = \overline{\partial_{M,r}^m\Gamma}$ is minimal and strongly proximal. For instance,

\begin{enumerate}
\item when $\Gamma=\Gamma_1 * ... *\Gamma_n$ is a free product of infinite groups $\Gamma_1,..., \Gamma_n$ for $n\geq 2$, and $\mu$ is a finitely supported, admissible, adapted and symmetric generating measure on this free product, by \cite[Theorem 26.1]{Woe00} the preimages of conical points are dense in the Martin boundary, and hence so are minimal points in $\partial_{M,r} \Gamma$.
\item when $\Gamma$ is a non-elementary hyperbolic, all points in $\partial_{M,r}\Gamma$ are automatically minimal, because they are all conical.
\item More generally, when $\Gamma$ is non-elementary hyperbolic relative to a finite collection of \emph{virtually abelian} subgroups, by \cite[Theorem~1.4]{DG21}, all points in $\partial_{M,r}\Gamma$ are minimal.
\end{enumerate}

\section{Spectral non-degeneracy} \label{s:spect-non-deg}
In this section we discuss spectral non-degeneracy, which is one of our standing assumptions for random walks on relatively hyperbolic groups. It roughly means that the $R$-induced first-return time random walk on a parabolic subgroup has spectral radius greater than one. Our goal is to show that spectrally non-degenerate random walks are ubiquitous, by showing that they exist on a variety of relatively hyperbolic groups.

Let $\Gamma$ be be hyperbolic relative to a finite collection of subgroups $\Omega$ group and let $\mu$ be a finitely supported, admissible and symmetric probability measure on $\Gamma$.
For a parabolic subgroup $H \in \Omega$, we consider the first return kernel $P_{r,H}$ to $H$ associated with the measure $r\mu$. That is, for $x,y$ in $ H$, we set
$$P_{r,H}(x,y)=\sum_{n\geq 1}\sum_{z_1,...,z_{n-1}\notin H}r^n\mu(x^{-1}z_1)...\mu(z_{n-1}^{-1}y).$$
We denote by $G_{r,H}$ the Green function associated with $P_{r,H}$ and by $R_H(r)$ the inverse of the spectral radius of $P_{r,H}$.
By \cite[Lemma~4.4]{DG21}, for every $r\leq R$, for every $x,y\in H$,
\begin{equation}\label{Greencoincide}
    G(x,y|r)=G_{r,H}(x,y|1).
\end{equation}
In particular, since $G(x,y|R)$ is finite, for every $r\leq R$, $R_{H}(r)\geq 1$.

\begin{Def}\label{defspectraldegeneracy}
Let $\Gamma$ be be hyperbolic relative to a finite collection of subgroups $\Omega$ group and let $\mu$ be a finitely supported, admissible and symmetric probability measure on $\Gamma$. We say that the random walk determined by $\mu$ is \emph{spectrally degenerate} along a parabolic subgroup $H \in \Omega$ if $R_H(R)=1$. 
We say that the random walk determined by $\mu$ is \emph{spectrally non-degenerate} if it is not spectrally degenerate along any $H \in \Omega$.
\end{Def}

Spectral non-degeneracy is of particular importance in the study of random walks on relatively hyperbolic groups.
For instance, it is a determining property characterizing the homeomorphism type of the $R$-Martin boundary as shown in \cite{DG21}, as well as a key property for establishing local limit theorems \cite{Dus22b}. Regardless of whether or not $\mu$ is spectrally non-degenerate, it follows from equation \eqref{Greencoincide} and \cite[Lemma~6.2, Proposition~6.3]{DG21}, respectively, that
\begin{equation} \label{eq:fin-para-green-der}
G_{R,H}(e,e|1) < \infty \ \text{and} \ \frac{d}{dt}_{|t=1}G_{R,H}(e,e|t)< \infty.
\end{equation}

We now restrict our attention to free products, which are particular cases of relatively hyperbolic groups.
Following the terminology of \cite{Woe00}, a probability measure $\mu$ on $\Gamma_1*\Gamma_2$ is called \textit{adapted} if it can be written as
$$\mu=\alpha_1\mu_1+\alpha_2\mu_2,$$
where $\alpha_i\geq 0$, $\alpha_1+\alpha_2=1$ and $\mu_i$ are probability measures on $\Gamma_i$ for $i=1,2$.

We assume that each $\mu_i$ are finitely supported, admissible and symmetric on $\Gamma_i$ for $i=1,2$.
We let $R_i$ be the inverse of the spectral radius of $\mu_i$ and $G_i$ be the Green function associated with $\mu_i$. For each $i=1,2$, by \cite[Proposition~9.18]{Woe00} there exists a continuous function $\zeta_i$ of $r$ such that for $x,y\in \Gamma_i$, for $r\leq R$,
\begin{equation}\label{formulaWoess}
\frac{G(x,y|r)}{G(e,e|r)}=\frac{G_{i}(x,y|\zeta_i(r))}{G_{i}(e,e|\zeta_i(r))}.
\end{equation}
Moreover, $\zeta_i(R)\leq R_i$ and $\zeta_i(r)<R_i$ if $r<R$.
We also define
\begin{equation*}
\theta=RG(e,e|R),\ \theta_i=R_iG_i(e,e|R_i)
\end{equation*}
and
\begin{equation}\label{defthetabar}
\overline{\theta}=\min \big \{\frac{\theta_1}{\alpha_1}, \frac{\theta_2}{\alpha_2} \big \}.
\end{equation}

We define implicitly the function $\Phi$ as
\begin{equation}\label{defPhi}
G(e,e|r)=\Phi(rG(e,e|r)),
\end{equation}
which is analytic on an open neighborhood of the interval $(0,\theta)$
and we consider the function $\Psi$ defined for $s<\theta$ as
\begin{equation}\label{defPsi}
\Psi(s)=\Phi(s)-t\Phi'(s).
\end{equation}
Since $G(e,e|R)$ is finite, the functions $\Phi$ and $\Psi$ are well-defined on $(0,\theta]$.

We define similarly the functions $\Phi_i$ and $\Psi_i$ associated with the probability measures $\mu_i$.
By \cite[Theorem 9.19]{Woe00} we get that
\begin{equation}\label{equationPhiPsi}
\begin{split}
\Phi(s) =\Phi_1(\alpha_1s) +\Phi_2(\alpha_2s) - 1 \\
\Psi(s) = \Psi_1(\alpha_1s) + \Psi_2(\alpha_2s) -1,
\end{split}
\end{equation}

This shows that both functions $\Phi$ and $\Psi$ can be extended to the interval $(0,\overline{\theta}]$.
The situation where $\Psi(\overline{\theta})<0$ is called the "typical case" in \cite{Woe00}.
This is equivalent to the fact that $\zeta_i(R)<R_i$ for $i=1,2$.
By \cite[Proposition~2.9]{Dus22b}, this corresponds to the case of a spectrally non-degenerate random walk.
Thus, in order to prove the following Proposition, we need to construct an example of an adapted random walk where $\Psi(\overline{\theta})<0$.

\begin{Prop}\label{propconstructionnonspecdeg}
Let $\Gamma=\Gamma_1*\Gamma_2$ be a non-elementary free product of two groups $\Gamma_1,\Gamma_2$.
Then, there exists a symmetric adapted probability measure $\mu$ with finite support generating $\Gamma$, which is spectrally non-degenerate.
\end{Prop}

\begin{proof}
We have two cases to consider.
Either one of the free factors $\Gamma_i$ is finite or both are infinite.

Let us first assume that $\Gamma_1$ is finite and by contradiction that $\mu$ is spectrally degenerate along $\Gamma_1$.
Then, $P_{R,\Gamma_1}$ is a $\Gamma_1$-invariant transition kernel whose spectral radius is 1.
Since $\Gamma_1$ is finite, the irreducible transition kernel $P_{R,\Gamma_1}$ must be recurrent, so in particular we get that
$G_{R,\Gamma_1}(e,e|1)$ is infinite, yielding a contradiction with equation \eqref{eq:fin-para-green-der}.
Thus, for any choice of $\alpha_1,\alpha_2$, the measure $\mu$ cannot be spectrally degenerate along $\Gamma_1$ when the latter is finite.
Also, we deduce from \cite[Equation (9.20)]{Woe00} that 
$$\zeta_i(R)G_i(e,e|\zeta_i(R))=\alpha_iRG(e,e|R)=\alpha_i\theta.$$
As $\alpha_2 \rightarrow 0$ we get $\alpha_1 \rightarrow 1$, so by \cite[Theorem 9.19]{Woe00} we have
$$\zeta_2(R) \leq \zeta_2(R)G_2(e,e|\zeta_2(R)) = \alpha_2\theta \leq \alpha_2\overline{\theta} \leq \alpha_2 \theta_1/\alpha_1 \to 0.$$
hence for $\alpha_2$ small enough we may arrange that $\zeta_2(R)<R_2$, so that $\mu$ is not either spectrally degenerate along $\Gamma_2$.

Next, without loss of generality we assume that both $\Gamma_1$ and $\Gamma_2$ are infinite.
Then, notice that both $\Gamma_1$ and $\Gamma_2$ contain an element of order at least 3.
Indeed, if without loss of generality we assume towards contradiction that $\Gamma_1$ does not, then every element of $\Gamma_1$ has order 2, which forces $\Gamma_1$ to be abelian. As $\Gamma_1$ is finitely generated, we must have that $\Gamma_1=(\mathbb{Z}/2\mathbb{Z})^k$. This implies that $\Gamma_1$ is finite, which is a contradiction. Thus, both $\Gamma_1$ and $\Gamma_2$ contain an element of order at least 3.
Now, by \cite[Lemma~17.9]{Woe00}, for each $i=1,2$ there exists a probability measure $\mu_i$ on $\Gamma_i$ such that
$\Psi_i(\theta_i)<1/2$.
Fixing $(\alpha_1,\alpha_2)$ such that $\theta_1/\alpha_1=\theta_2/\alpha_2=\overline{\theta}$, we have by equation \eqref{equationPhiPsi} that $\Psi(\overline{\theta})<0$, so that $\mu$ is spectrally non-degenerate.
\end{proof}

Our next goal is to provide examples where spectral non-degeneracy is automatic for relatively hyperbolic groups with virtually nilpotent parabolic subgroups of low homogeneous dimension.

By a celebrated result of Gromov \cite{Gro81}, finitely generated virtually nilpotent groups are exactly discrete groups of polynomial volume growth. That is, balls grow asymptotically like $n^d$. Bass \cite{Bas72} and Guivarc'h \cite{Gui73} independently identified $d$ as the homogeneous dimension of $\Gamma$.

\begin{Def}\label{defrankanddimension}
Let $\Gamma$ be a nilpotent group. Let
$\Gamma_1=\Gamma$ and $\Gamma_n=[\Gamma_{n-1},\Gamma]$. Let $N_{\Gamma}$ be the nilpotency class of $\Gamma$, which is the largest integer such that $\Gamma_N$ is not trivial. 
\begin{enumerate}
    \item The \emph{rank} of $\Gamma$ is given by
$$
\mathrm{rank}(\Gamma)=\sum_{k=1}^{N_\Gamma}\mathrm{rank}\left (\Gamma_k/\Gamma_{k+1}\right).
$$

    \item The \emph{homogeneous dimension} of $\Gamma$ is given by
$$
    d=\sum_{k=1}^{N_{\Gamma}}k \cdot \mathrm{rank} \big(\Gamma_k/\Gamma_{k+1}\big ).
$$
\end{enumerate}
When $\Gamma$ is virtually nilpotent, we define its homogeneous dimension as the homogeneous dimension of a finite index nilpotent subgroup.
\end{Def}

Note that when $\Gamma$ is nilpotent, all groups $\Gamma_k/\Gamma_{k+1}$ are finitely generated abelian groups and must therefore have a well defined rank. Moreover, when $\Gamma$ is virtually nilpotent the above definitions turn out to be independent of the finite index nilpotent subgroup. This can be proved by direct computations, but this also follows from the fact that $d$ is the degree of the growth of balls, which is invariant under quasi-isometry, and therefore invariant up to finite index.

In \cite[Proposition~6.1]{DG21}, it is proved that if $\Gamma$ is hyperbolic relative to virtually \emph{abelian} parabolic subgroups of rank at most $4$,
then every finitely supported, admissible and symmetric probability measure is spectrally non-degenerate. We extend this result to relatively hyperbolic groups with respect to virtually \emph{nilpotent} parabolic subgroups of homogenous dimension at most $4$.

\begin{Prop}\label{propsmallhomogeneousdimension}
Let $\Gamma$ be hyperbolic relative to a finite collection of virtually nilpotent parabolic subgroups $\Omega$, and let $\mu$ be a finitely supported, admissible and symmetric probability measure on $\Gamma$. Let $H \in \Omega$ be virtually nilpotent of homogeneous dimension at most 4. Then, $\mu$ is not spectrally degenerate along $H$.
\end{Prop}

\begin{proof}
Let $\eta\geq0$. We consider the first return kernel $P_{R,H,\eta}$ to the $\eta$-neighbor\-hood $N_\eta(H)$ of $H$, associated with the measure $R\mu$.
We let $G_{R,H,\eta}$ be the corresponding Green function. By assumption, $H$ is virtually nilpotent.
As in \cite[Section~4]{DG21}, we may identify the $\eta$-neighborhood of $H$ as $\mathcal{N}\times \{1...,N_\eta\}$, where $\mathcal{N}$ is a finite index nilpotent subgroup of $H$.
We then write $\mathcal{A}$ for the abelianization of $\mathcal{N}$ which is thus a finitely generated abelian group of rank $d_\mathcal{A}$ and we set
$\pi$ the projection from $\mathcal{N}$ to $\mathbb Z^{d_\mathcal{A}}$.
Following \cite{DG21}, for fixed $u\in \mathbb R^{d_\mathcal{A}}$, and $j,k\in \{1,...,N_\eta\}$, we define
$$F_{j,k}(u)=\sum_{x\in \mathcal{N}}P_{R,H,\eta}((e,j),(x,k))\mathrm{e}^{\pi(x)\cdot u}.$$
The matrix $F(u)$ with entries $F_{j,k}(u)$ is irreducible and has a dominant eigenvalue that we denote by $\lambda(u)$.
We also consider a left eigenvector $\nu(u)$ and a right eigenvector $C(u)$ associated with $\lambda(u)$ and we normalize them by declaring that $\nu(u)\cdot C(u)=1$.

By definition, spectral degeneracy along $H$ means that for $\eta=0$, the spectral radius of $P_{R,H,\eta}$ is 1.
By \cite[Lemma~4.9]{DG21}, this implies the same thing for every positive $\eta$.
Consequently, by \cite[Lemma~4.8]{DG21}, the minimum of the function $\lambda(u)$ is 1.
Also by \cite[Proposition~4.10]{DG21}, the 1-Martin boundary of $P_{R,H,\eta}$ is reduced to a point, and this implies by \cite[Lemma~4.5]{DG21} that the set of $u$ such that $\lambda(u)=1$ is reduced to a point.
Since the initial random walk determined by $\mu$ is symmetric, we see that $F(u)$ is symmetric and so the minimum of $\lambda$ is necessarily reached at $u=0$.
We conclude that
\begin{equation}\label{equationeigenvectors}
    \nu(0)F(0)C(0)=1.
\end{equation}

Now define the averaged transition kernel on $\mathcal{N}$ by setting
$$\tilde{P}(x,y):=\sum_{j,k}\nu_j(0) P_{R,H,\eta}((x,j),(y,k))C_k(0).$$
Equation~(\ref{equationeigenvectors}) can then be reformulated as
$$\sum_{x\in \mathcal{N}}\tilde{P}(e,x)=1.$$
Moreover, as the initial random walk determined by $\mu$ is invariant by a subgroup action, we get that the first return kernel
$P_{R,H,\eta}$ is $\mathcal{N}$-invariant, and hence so is $\tilde{P}$.
Therefore, $\tilde{P}$ is a random walk on $\mathcal{N}$ which is determined by some symmetric probability measure $\mu_{\mathcal{N}}$ on $\mathcal{N}$.

By \cite[Lemma~4.6]{DG21}, if $\eta$ is chosen large enough, then $P_{R,H,\eta}$ has exponential moments.
We fix such an $\eta$. It follows that $\tilde{P}$ also has exponential moments, since the coordinates of $\nu(0)$ and $C(0)$ are finite.
Finally, by \cite[Proposition~6.3]{DG21} we get that,
$$\sum_{j,k}\sum_{x\in \mathcal{N}}G_{R,H,\eta}((e,j),(x,k)|1)G_{R,H,\eta}((x,k),(e,j)|1)$$
is finite.
We let $G_{R,H,\eta}(x,y|1)$ be the matrix whose $(j,k)$ entry is given by $G_{R,H,\eta}((x,j)(y,k)|1)$.
Denoting by $\widetilde G$ the Green function associated with $\tilde{P}$, we get that
$$\sum_{x\in \mathcal{N}}\widetilde G(e,x|1)\widetilde G(x,e|1)=\sum_{x\in \mathcal{N}}\nu\cdot G_{R,H,\eta}(e,x|1)\cdot C \cdot \nu \cdot G_{R,H,\eta}(x,e|1)\cdot C.$$
Recalling that $\nu$ and $C$ are normalized so that $\nu\cdot C=1$ and that they have bounded coordinates, we get that
$$\sum_{x\in \mathcal{N}}\widetilde G(e,x|1)\widetilde G(x,e|1)$$
is finite.
By \cite[Lemma~6.2]{DG21}, we get that $\frac{d}{dt}|_{t=1} \widetilde G(e,e|t)$ is finite as well.

To summarize, we constructed a symmetric probability measure $\mu_{\mathcal{N}}$ whose support (which is not necessarily finite) generates $\mathcal{N}$ as a semigroup, where $\mathcal{N}$ is a finite index nilpotent subgroup of $H$ with exponential moments, and such that the associated Green function has finite derivative at $1$.
Let $\nu_{H}$ be an auxiliary finitely supported, admissible, symmetric probability measure on $H$.
By Pittet and Saloff-Coste comparison theorems \cite{PS-C00} (see also \cite[Theorem~15.1]{Woe00}), we have
$$\mu_{\mathcal{N}}(e)^{*n}\succeq \nu_{H}^{*n}(e),$$
which means that there exist $C_1,C_2,C_3>0$ such that sufficiently large $n\in \mathbb{N}$ we have,
$$\mu_{\mathcal{N}}^{*n}(e)\geq C_1 \sup\big\{\nu_{H}^{*k}(e),C_2n \leq k \leq C_3n\big\}.$$
Since $\nu_H$ has finite support, by a result of Alexopoulos \cite[Corollary 1.17]{Alexopoulos}, there exists $C>0$ such that $\nu_{H}^{*k}(e) \sim Ck^{-d/2}$ as $k\rightarrow \infty$. Thus, we find that
$$\mu_{\mathcal{N}}^{*n}(e)\geq C' n^{-d/2}$$
for some constant $C'$,
where $d$ is the homogeneous dimension of $H$.
When $d\leq 4$, this implies that the derivative at 1 of the Green function $t\mapsto \widetilde G(e,e|t)$ is infinite, arriving at a contradiction. Hence, we see that $\mu$ is not spectrally degenerate.
\end{proof}

Let $M$ be a geometrically finite Riemannian manifold of pinched negative curvature. Then, the fundamental group $\pi_1(M)$ is relatively hyperbolic with respect to the cusp stabilizers, see \cite{Bow95} and \cite{Bow12}.

\begin{Cor}
Let $M$ be a geometrically finite Riemannian manifold of pinched negative curvature and let $\mu$ be a finitely supported, admissible and symmetric probability measure on $\pi_1(M)$.
Assume that either
\begin{enumerate}
    \item $M$ has constant negative curvature and $\mathrm{dim}( M)\leq 5$; or,
    \item $\mathrm{dim}(M)\leq 4$.
\end{enumerate}
Then $\mu$ is a spectrally non-degenerate random walk on $\pi_1(M)$.
\end{Cor}

\begin{proof}
There are two cases. Either $M$ has constant negative curvature, or variable negative curvature. If $M$ has constant negative curvature, the case is treated as in \cite[Theorem~1.5]{DG21}. The result follows from the fact that horospheres can be given a Euclidean structure of dimension at most $\mathrm{dim}(M)-1$. Assuming that $\mathrm{dim} (M)\leq 5$, this implies that cusp stabilizers are virtually abelian of rank at most $4$ (see \cite[Chapter~5]{Rat06} for more details).

Thus, let us consider the case of variable negative curvature. Let $H$ be a parabolic subgroup of $\pi_1(M)$. The assumption that $M$ is geometrically finite implies that $H$ acts co-compactly on horospheres centered at the parabolic point fixed determining $H$.
By \cite[Main Theorem, Section~1.5]{BK81}, $H$ contains a finite index torsion free nilpotent group $\Gamma$ of rank $\mathrm{dim}(M)-1$.

Let $N_\Gamma$ be the nilpotency class of $\Gamma$ and recall that we defined the rank of a virtually nilpotent group in Definition~\ref{defrankanddimension}.
Let us prove that whenever $\mathrm{rank}(\Gamma)\leq 3$, the homogeneous dimension of $\Gamma$ is at most $4$. First, we may assume that $\Gamma$ is not virtually cyclic, for otherwise its homogeneous dimension is at most $1$. Then, according to \cite[Lemma~6.6]{DGdrift}, the rank of $\Gamma_1/\Gamma_2$ is at least 2.
We first deduce that $N_\Gamma$ is at most 2.
Furthermore, if $\mathrm{rank}(\Gamma)=2$, then we necessarily have that $\Gamma$ is virtually $\mathbb Z^2$ and its homogeneous dimension is then also $2$.
Assuming now that $\mathrm{rank}(\Gamma)=3$, we only have two possibilities. Either $\mathrm{rank}(\Gamma_1/\Gamma_2)=3$, in which case $\Gamma$ is virtually $\mathbb Z^3$ and has homogeneous dimension $3$, or $\mathrm{rank}(\Gamma_1/\Gamma_2)=2$ and $\mathrm{rank}(\Gamma_2/\Gamma_3)=1$, in which case the homogeneous dimension is $4$.
\end{proof}

\begin{Exl}
In case where $M$ has varying negative curvature, the given bound $\mathrm{dim}\ M\leq 4$ is optimal to ensure that the homogeneous dimension of parabolic subgroups is at most $4$. Indeed, consider the discrete Heisenberg group $\mathrm{H}_3(\mathbb Z)$, and consider $\Gamma=\mathrm{H}_3(\mathbb Z)\times \mathbb Z$.
Then, $\Gamma$ is nilpotent and has homogeneous dimension $5$.
Furthermore, $\Gamma$ is a lattice in the nilpotent Lie group $\mathcal{N}=\mathrm{H}_3(\mathbb R)\times \mathbb R$, which has topological dimension $4$.
Now, $\mathcal{N}/\Gamma$ is a nilmanifold, so by \cite[Corollary~6]{Ont20} we get that $\Gamma$ can be realized as a cusp stabilizer of a pinched negatively curved manifold of finite volume and of dimension $5$.

When negative curvature is constant, the bound $\mathrm{dim}(M)\leq 5$ is also optimal, as for any finite volume non-compact hyperbolic manifold of dimension $6$, the cusp stabilizers are virtually $\mathbb Z^5$.
\end{Exl}

Thus, we see that spectrally non-degenerate random walks on relatively hyperbolic are abundant. In \cite{Woe00}, spectrally non-degenerate adapted random walks on free products are referred to as the ``typical case". We conjecture that if $\Gamma$ is relatively hyperbolic, and if we take $\mu_n$ as the uniform measure on $n$-balls $B_n$ (around the identity element $e$) in $\Gr(\Gamma)$, then for large enough $n$ the random walk determined by $\mu_n$ will be spectrally non-degenerate. 

\section{Asymptotics of Ratio limits} \label{s:asymp}

Suppose that $\Gamma$ is a non-elementary hyperbolic relative to a finite collection of subgroups $\Omega$, and that $\mu$ is a finitely supported, admissible, aperiodic and symmetric probability measure on $\Gamma$. To analyze the ratio-limit boundary for random walks we will need to know the asymptotic behavior of derivatives of Green function as $r\rightarrow R$. This will lead us to a formula for the ratio-limit kernels which involves the following iterated Green sums. 

Following \cite{Dus22a, DPT+}, we say that the random walk determined by $\mu$ is convergent if $\lim_{r\rightarrow R}G'(e,e|r)$ exists. This occurs if and only if $\lim_{r\rightarrow R}G'(x,y|r)$ exists for any $x,y \in \Gamma$. If $\mu$ is not convergent, we will say that it is divergent. When $\Gamma$ is hyperbolic relative to virtually abelian subgroups and $\mu$ is convergent, we define its \textit{spectral degeneracy rank} $d$ as the smallest rank of a parabolic subgroup along which $\mu$ is spectrally degenerate.

For $x,y\in \Gamma$ let us denote
$$
I^{(s)}(x,y|r) := \sum_{x_1,...,x_s \in \Gamma} G(x,x_1|r)G(x_1,x_2|r) ... G(x_{s-1},x_s|r)G(x_s,y|r).
$$

By \cite[Proposition 5.2]{DPT+}, if there is $s \in \mathbb{N}$ which is the smallest integer with $\underset{r \rightarrow R}{\lim} G^{(s)}(x,y|r) = \infty$ (or equivalently $\underset{r \rightarrow R}{\lim} I^{(s)}(x,y|r) = \infty$), then there is a constant $C>0$ independent of $x,y\in \Gamma$ so that
$$
I^{(s)}(x,y|r) \underset{r \rightarrow R}{\sim} C \cdot G^{(s)}(x,y|r).
$$

To arrive at a formula for the ratio-limit kernels involving the expressions $I^{(s)}(x,y|r)$, we will need to relate the asymptotic behavior for higher derivatives of Green functions with that of $P^n(x,y)$. This was established by the second author and collaborators in a sequence of papers \cite{Dus22a, Dus22b, DPT+}.

\begin{enumerate}
\item When the random walk is spectrally non-degenerate, by \cite[Corollary~7.1]{Dus22b}, we know that $s=1$ is the smallest integer such that we have $\underset{r \rightarrow R}{\lim} G^{(s)}(x,y|r) = \infty$ and for any $x,y \in \Gamma$ there exists $\beta(x,y) >0$ such that
\begin{equation*}
G'(x,y|r) \underset{r\rightarrow R}{\sim} \frac{\beta(x,y)}{\sqrt{R - r}}; \ \ \text{or,} 
\end{equation*}

\item when $\Gamma$ is hyperbolic relative to virtually abelian subgroups and $\mu$ convergent, with spectral degeneracy rank $d$, by \cite[Corollary~6.2]{DPT+}, we know that $s = \lceil \frac{d}{2} \rceil -1$ is the smallest integer such that we have $\underset{r \rightarrow R}{\lim} G^{(s)}(x,y|r) = \infty$ and for any $x,y \in \Gamma$ there exists $\beta(x,y) >0$ such that either
\begin{equation} \label{eq:sqrt}
G^{(s)}(x,y|r) \underset{r\rightarrow R}{\sim} \frac{\beta(x,y)}{\sqrt{R - r}}, \ \text{if } d \ \text{is odd, or} 
\end{equation}
\begin{equation} \label{eq:log}
G^{(s)}(x,y|r) \underset{r\rightarrow R}{\sim} \beta(x,y) \cdot \log \Big(\frac{1}{{R - r}} \Big)\ \text{if } d \ \text{is even}.
\end{equation}
\end{enumerate}

In the case where equation \eqref{eq:sqrt} is satisfied, following the calculations in \cite[Section~9]{GL13} (or equivalently those that are based on Karamata's Tauberian theorem \cite[Corollary~1.7.3]{BGT87}), there is a constant $C>0$ independent of $x,y \in \Gamma$ such that
\begin{equation} \label{eq:ll}
P^{n}(x,y) \underset{n \rightarrow \infty}{\sim} C \cdot \beta(x,y) R^{-n}n^{-\frac{d}{2}}.
\end{equation}

In the case where equation \eqref{eq:log} is satisfied, following the calculation performed after \cite[Corollary 6.1]{DPT+}), we also get a constant $C>0$ independent of $x,y\in \Gamma$ such that equation \eqref{eq:ll} holds.

Since we need to know that the constant $C>0$ is independent of $x,y\in \Gamma$, let us see how we obtain equation \eqref{eq:ll} from equation \eqref{eq:log}. Indeed, taking the slowly varying function ``$\log$" in \cite[Corollary 1.7.3]{BGT87} we get
$$
\sum_{k=0}^n k^sR^kP^{k}(x,y) \underset{n \rightarrow \infty}{\sim} R^s \cdot \beta(x,y) \cdot \log(n)
$$
By \cite[Corollary 9.4]{GL13}, for $x\neq y$ there exist $\delta >0$ and a non-increasing sequences $(q_n(x,y))_n$ such that
$$
R^nP^{n}(x,x) = q_n(x,x) + O(e^{-\delta n}), \ \text{and}
$$
$$
R^n(P^{n}(x,x) + P^{n}(x,y)) = q_n(x,y) + O(e^{-\delta n}).
$$
Thus, we deduce that
$$
\sum_{k=0}^n k^s q_k(x,x) \underset{n \rightarrow \infty}{\sim} R^s \beta(x,y)  \log(n)
$$
Hence, by the same proof as in \cite[Lemma 9.5]{GL13} (with $\beta =1$) we get that
$$
n^sq_n(x,x) \underset{n \rightarrow \infty}{\sim} R^s \beta(x,x)n^{-1},
$$
so we get that
$$
P^{n}(x,x) \underset{n \rightarrow \infty}{\sim} \beta(x,x) R^{s-n} n^{-s-1}.
$$

Now, since 
$$
G^{(s)}(x,x|r) + G^{(s)}(x,y|r) \underset{r \rightarrow R}{\sim} (\beta(x,x) + \beta(x,y)) \log \Big(\frac{1}{R -r} \Big),
$$
we may again apply \cite[Corollary 1.7.3]{BGT87} to get that
$$
\sum_{k=0}^n k^s q_k(x,y) \underset{n \rightarrow \infty}{\sim} (\beta(x,x) + \beta(x,y)) \log (n),
$$
so that again the same proof as in \cite[Lemma 9.5]{GL13} (with $\beta =1$) implies that 
$$
n^s q_n(x,y) \underset{n \rightarrow \infty}{\sim} R^s \cdot (\beta(x,x) + \beta(x,y))n^{-1}.
$$
Thus, 
$$
P^{n}(x,x) + P^{n}(x,y) \underset{n \rightarrow \infty}{\sim} (\beta(x,x) + \beta(x,y))R^{s-n}n^{-s-1},
$$
so, we finally get that for any $x,y\in \Gamma$
\begin{equation*}
P^{n}(x,y) \underset{n \rightarrow \infty}{\sim} \beta(x,y) R^{s-n}n^{-\frac{d}{2}}.
\end{equation*}
Hence, we obtain equation \eqref{eq:ll}. Such local limit theorems are then used to obtain the following result.

\begin{Cor} \label{c:ratio-limit-iterated-sums}
Let $\Gamma$ be a finitely generated non-elementary relatively hyperbolic group, and assume that $\mu$ is a finitely supported, admissible, aperiodic and symmetric probability measure on $\Gamma$. Suppose either that 
\begin{enumerate}
\item the random walk is spectrally non-degenerate, in which case we take $s=1$ or;
\item that $\Gamma$ is hyperbolic relative to virtually abelian subgroups, with $\mu$ convergent, with spectral degeneracy rank $d$, in which case we take $s = \lceil \frac{d}{2} \rceil -1$.
\end{enumerate}
Then, the ratio-limit kernel satisfies
$$
H(x,y) = \lim_{r\rightarrow R} \frac{I^{(s)}(x,y|r)}{I^{(s)}(e,y|r)}.
$$
\end{Cor}

\begin{proof}
By combining equations \eqref{eq:sqrt} and \eqref{eq:ll} in the first case, and either \eqref{eq:sqrt} and \eqref{eq:ll} when $d$ is odd or \eqref{eq:log} and \eqref{eq:ll} when $d$ is even in the second case, we get that
\begin{equation*}
    H(x,y) = \frac{\beta(x,y)}{\beta(e,y)} = \lim_{r\rightarrow R}\frac{I^{(s)}(x,y|r)}{I^{(s)}(e,y|r)}. \qedhere
\end{equation*}
\end{proof}

In \cite[Question 7.1(a)]{Woe21}, it is asked under which conditions the conclusion of Corollary \ref{c:ratio-limit-iterated-sums} holds with $s=1$. We give the following sufficient condition.

\begin{Prop} \label{p:limit-harmonic}
Let $\Gamma$ be a finitely generated non-elementary relatively hyperbolic group, and assume that $\mu$ is a finitely supported, admissible and symmetric probability measure on $\Gamma$. Let $0<s\in \mathbb{N}$ be so that for any $0\leq k<s$ we have that $\lim_{r\rightarrow R}G^{(s)}(x,y|r)$ exists. Then, for all $0\leq k<s$ the function $\mathfrak{H}_k$ given for $x,y\in \Gamma$ by
$$
\mathfrak{H}_k(x,y):= \lim_{r\rightarrow R} \frac{I^{(k)}(x,y|r)}{I^{(k)}(e,y|r)}
$$
is not $R^{-1}$-harmonic. In particular, if there exists $s\in \mathbb{N}$ which is the smallest integer satisfying $H(x,y) = \mathfrak{H}_s(x,y)$ (for some non-principle ultrafilter $\omega$), then $\lim_{r\rightarrow R}G^{(s)}(x,y|r) = \infty$.
\end{Prop}

\begin{proof}
Define inductively 
$$
F_1(x,y|r) :=\frac{d}{dr}(rG(x,y|r)),
$$
and for $k\geq 2$, 
$$
F_k(x,y|r):= \frac{d}{dr}(r^2F_{k-1}(x,y|r)).
$$
By \cite[Lemma 3.2]{Dus22a} we have that $F_k(x,y|r) = k! r^{k-1}I^{(k)}(x,y|r)$. Thus, we have for all $0\leq k < s$ that
$$
\mathfrak{H}_k(x,y) = \frac{F_k(x,y|R)}{F_k(x,y|R)}.
$$
We will prove by induction on $k$, for $0\leq k <s$, that for any $y\in \Gamma$ the function $x\mapsto \mathfrak{H}_k(x,y)$ is $R^{-1}$-superharmonic but not $R^{-1}$-harmonic.

For the base case $k=0$ we get that $\mathfrak{H}_0(x,y) = K_R(x,y) = \frac{G(x,y|R)}{G(e,y|R)}$ is the Martin kernel, and a straightforward calculation shows that for each $y\in \Gamma$, the function $x\mapsto G(x,y|R)$ is $R^{-1}$-superharmonic but not $R^{-1}$-harmonic at $y$. Thus, $\mathfrak{H}_0(x,y)$ is $R^{-1}$-superharmonic but not $R^{-1}$-harmonic.

Now let $0< k < s-1$ and assume that $\mathfrak{H}_k$ is $R^{-1}$-superharmonic but not $R^{-1}$-harmonic. Then, since $F_k(x,y|r)$ only involves Green derivatives of order at most $k$, its derivative $\frac{d}{dr}|_{r=R}(F_k(x,y|r))$ at $r=R$ would involve Green derivatives of order at most $k+1$, and hence converges as $r\rightarrow R$. Thus, $F'_k(x,y|R)$ is well-defined and is still $R^{-1}$-superharmonic in $x$. By inductive definition we get
$$
F_{k+1}(x,y|R) - R\sum_{z\in \Gamma}P(x,z)F_{k+1}(z,y|R) = 
$$
$$
2R\Big[ F_k(x,y|R) - R\sum_{z\in \Gamma}P(x,z)F_k(z,y|R)\Big] + 
$$
$$
R^2\Big[ F'_k(x,y|R) - R\sum_{z\in \Gamma}P(x,z)F'_k(z,y|R)\Big].
$$

Since $F_k(x,y|R)$ and $F'_k(x,y|R)$ are $R^{-1}$-superharmonic in $x$, the above expression is non-negative, but $F_k(x,y|R)$ is not $R^{-1}$-harmonic in $x$, so we get that $F_{k+1}(x,y|R)$ is also $R^{-1}$-superharmonic but not $R^{-1}$ harmonic in $x$. Since all expressions involved in the definition of $x\mapsto \mathfrak{H}_{k+1}(x,y)$ are finite, we see that for all $y\in\Gamma$, this function is not $R^{-1}$-harmonic.

Finally, if $s\in \mathbb{N}$ is the smallest integer such that $H(x,y) = \mathfrak{H}_s(x,y)$, since $x\mapsto H(x,y)$ is $R^{-1}$-harmonic, we must have that $\lim_{r\rightarrow R}G^{(s)}(x,y|r) = \infty$ by the above.
\end{proof}

As we will see, there are many examples with $s=1$ for which the conclusion of Corollary~\ref{c:ratio-limit-iterated-sums} can fail, even for free products of abelian groups. For instance, consider $\Gamma = \Gamma_1*\Gamma_2$ with a convergent probability measure $\mu$ on it. That is, for $x,y \in \Gamma$ we have that $\lim_{r\to R} G(x,y|r)$ and $\lim_{r\to R}I^{(1)}(x,y|r)$ exist for all $x,y\in \Gamma$. 
Such convergent measures $\mu$ do exist, see for instance \cite[Section~7, Example~D]{CG12} for examples with $\Gamma_i=\mathbb Z^{d_i}$. We will also construct such in Examples~\ref{exampleallderivativefinite} and \ref{ex:limitprobmeasures} below. Thus, by Proposition \ref{p:limit-harmonic}, we see that the conclusion of Corollary \ref{c:ratio-limit-iterated-sums} would fail in these examples for $s=1$.
As we saw, in many cases there is a smallest $s\in \mathbb{N}$ such that $I^{(s)}(x,y|r)$ tends to infinity as $r\to R$. Our next goal is to construct examples of random walks such that there is no such $s$, i.e.\ $\lim_{r\rightarrow R} G^{(s)}(x,y|r)$ exists for \emph{any} $s\in \mathbb{N}$. Using Proposition~\ref{p:limit-harmonic}, we will be able to deduce that the conclusion of Corollary~\ref{c:ratio-limit-iterated-sums} can fail for \emph{any} $s\in \mathbb{N}$. We will need the following result of Varapoulos. 

Let $\Gamma$ be a group of exponential volume growth and let $\mu$ be a finitely supported, admissible and symmetric probability measure on $\Gamma$. Then, by \cite[Theorem~1]{Var91} there are constants $c,C>0$ such that
$$
\mu^{*n}(e)\leq C\mathrm{e}^{-cn^{1/3}}.
$$
More generally (see \cite[Theorem 6.7]{Woe94}), if $n$-balls around $e$ in $\Gr(\Gamma)$ satisfy
\begin{equation}\label{equationgrowth}
|B_n| \geq \mathrm{e}^{cn^\alpha}
\end{equation}
for some $0<\alpha\leq 1$, then
\begin{equation} \label{eq:upperbound}
\mu^{*n}(e)\leq C\mathrm{exp}\left(-cn^{\frac{\alpha}{\alpha+2}}\right).
\end{equation}

\begin{Exl}\label{exampleallderivativefinite}
Consider a finitely generated amenable group $\Gamma_1$ and set $\Gamma_2=\mathbb{Z}$. Consider a finitely supported, admissible and symmetric probability measures $\mu_i$, and let $\mu$ be the adapted probability measure on $\Gamma=\Gamma_1*\Gamma_2$ given by
$$
\mu=\alpha_1\mu_1+\alpha_2\mu_2.
$$
Note that since both $\Gamma_1$ and $\Gamma_2$ are amenable, we have that the spectral radii of $\mu_1$ and $\mu_2$ are $R_1=R_2=1$.

Denote by $T^+_e$ the first return time to $e$ according to the random walk determined by $\mu$.
Let
$$U(r)=\sum_{n\geq 1} \mathbb P_e[T^+_e = n] r^n$$
be the generating function of first return time to $e$. Then, according to \cite[Lemma~1.13~(a)]{Woe00} we have that $G(e,e|r)(1-U(r))=1$.
Also, by \cite[Equation (9.14)]{Woe00}, the function $\Psi$ defined by equation \eqref{defPsi} satisfies
$$\Psi(t)=\frac{1}{1+ rU'(r)-U(r)}$$
for $t:=rG(e,e|r)\leq \theta=RG(e,e|R)$. Using $G(e,e|r)(1-U(r))=1$, we get
$$\Psi(t)=\frac{G(e,e|r)^2}{rG'(e,e|r)+G(e,e|r)}.$$
Since $\Gamma$ is non-amenable, we have that $G(e,e|R)$ is finite, so we get that
that $\Psi(\theta)=0$ if and only if $\lim_{r\rightarrow R}G'(e,e|r) = \infty$.

Our first goal is to prove that for a suitable choice of $(\alpha_1,\alpha_2)$, the random walk is convergent, or equivalently that $\Psi(\theta)\neq 0$.
Rewriting $\Psi$ as
$$\Psi(t)=\frac{1}{1+\sum_{n\geq 0} (n-1)\mathbb P_e[T^+_e = n] r^n},$$
it follows that $\Psi$ is decreasing.
Thus, it is enough to prove that for a suitable choice of $(\alpha_1,\alpha_2)$ we have $\Psi(\overline{\theta})>0$, where $\overline{\theta}$ is defined in equation \eqref{defthetabar}.

Note that in our example, since the random walk on $\Gamma_2=\mathbb Z$ determined by $\mu_2$ must be recurrent, we have $G_2(e,e|1)=\infty$.
In particular, we have $\overline{\theta}=\theta_1/\alpha_1$.
Therefore, by equation \eqref{equationPhiPsi}, where $\Psi_1,\Psi_2$ are defined similarly for $\mu_1,\mu_2$ respectively, we have,
$$
\Psi(\overline{\theta})=\Psi_1(\theta_1)+\Psi_2(\alpha_2\theta_1/\alpha_1)-1=\Psi_1(\theta_1)+\Psi_2\left (\frac{(1-\alpha_1)\theta_1}{\alpha_1}\right )-1.
$$

Suppose now that $\Gamma_1$ satisfies~(\ref{equationgrowth}).
By the upper bound in equation \eqref{eq:upperbound}, for every $s\geq 0$, we have
\begin{equation}\label{Gsfiniteparabolic}
G_1^{(s)}(e,e|1)<\infty.
\end{equation}
Consequently, for $s=1$, we get that $\Psi_1(\theta_1)>0$.
Now, as $\alpha_1$ tends to 1, $\frac{(1-\alpha_1)\theta_1}{\alpha_1}$ tends to 0 and since $\Psi_2(0)=1$, we see that
$\Psi(\overline{\theta})$ tends to $\Psi_1(\theta_1)$.
In particular, for small enough $\alpha_1$, $\Psi(\overline{\theta})>0$ and so the random walk determined by $\mu$ is convergent for that choice of $\alpha_1$.

Next, we show that $\lim_{r\to R}I^{(s)}(e,e|r)$ is finite for every $s$.
For $i=1,2$, we consider the iterated Green sums for the subgroups $\Gamma_i$,

$$I_{\Gamma_i}^{(s)}(e,e|r)=\sum_{x_1,...,x_s\in \Gamma_i}G(e,x_1|r)...G(x_s,e|r).$$
By \cite[Lemma~5.7]{Dus22a}, since $\lim_{r\to R}I^{(1)}(e,e|r)$ is finite, $I^{(s)}(e,e|r)$ is bounded by quantities only involving $I^{(k)}(e,e|r)$ for $k<s$ and $I_{\Gamma_i}^{(k)}(e,e|r)$ for $k\leq s$.
Thus, by induction, to show that $\lim_{r\to R}I^{(s)}(e,e|r)$ is finite for all $s$, it suffices to show that $\lim_{r\to R}I_{\Gamma_i}^{(s)}(e,e|r)$ is finite for all $s$.

First, let us deal with $\Gamma_1$.
Using the relation between the derivatives of the Green function and the iterated sums given by \cite[Proposition 5.2]{DPT+} on the one hand and the relation between the Green functions $G$ and $G_i$ given by equation \eqref{formulaWoess} on the other hand, there are constants $C_1,C_2>0$ such that
\begin{align*}
I_{\Gamma_1}^{(s)}(e,e|r)& =\left(\frac{G(e,e|r)}{G_1(e,e|\zeta_i(r))}\right)^{s+1}\sum_{x_1,...,x_s\in \Gamma_1}G_1(e,x_1|\zeta_1(r))...G_1(x_s,e|\zeta_1(r))\\
&\leq C_1 \sum_{x_1,...,x_s\in \Gamma_1}G_1(e,x_1|\zeta_1(R))...G_1(x_s,e|\zeta_1(R))\\
&\leq C_1 C_2 G_1^{(s)}(e,e|\zeta_1(R)).
\end{align*}
Thus, we need only prove that $G_1^{(s)}(e,e|\zeta_1(R))$ is finite for all $s$. This follows from equation \eqref{Gsfiniteparabolic}.

Second, since we have $\Gamma_2=\mathbb Z$, we deduce from Proposition \ref{propsmallhomogeneousdimension} that the random walk cannot be spectrally degenerate along $\Gamma_2$. By using equation \eqref{Greencoincide},  a similar calculation as above shows that we need only show that $G_{R,\Gamma_2}^{(s)}(e,e|1)$ is finite for every $s\geq 0$. But now, spectral non-degeneracy along $\Gamma_2$ implies that $1$ is smaller than the radius of convergence of the Green function $G_{R,\Gamma_2}(e,e|r)$, so we deduce again that
$\lim_{r\to R}I_{\Gamma_2}^{(s)}(e,e|r)$ is finite. 

Thus, we have that $\lim_{r\rightarrow R}I^{(s)}(e,e|r)$ exists for all $s\geq 0$. A standard calculation then shows that $\lim_{r\to R}I^{(s)}(x,y|r)$ exists for all $x,y\in \Gamma$ and for all $s\geq 0$. 
\end{Exl}

We also get the following examples of convergent random walks on free products of $\mathbb{Z}^5$ and $\mathbb{Z}$.

\begin{Exl}\label{ex:limitprobmeasures}
Let $\Gamma = \Gamma_1 * \Gamma_2$ where $\Gamma_1 = \mathbb{Z}^5$ and $\Gamma_2=\mathbb{Z}$. Suppose $\mu_i$ are finitely supported, admissible and symmetric probability measures on $\Gamma_i$ for $i=1,2$, and let $\mu := \alpha_1 \mu_1 + \alpha_2 \mu_2$ be an adapted probability measure on $\Gamma$. Then, we will show that for small enough $\alpha_1$, the random walk $\mu$ is spectrally non-degenerate, and for large enough $\alpha_1$, the random walk $\mu$ is spectrally degenerate and convergent along $\mathbb Z^5$.

Indeed, we only need to prove that $\Psi(\bar\theta)$ is negative for small enough $\alpha_1$ and positive for large enough $\alpha_1$. Following the same line of reasoning as in Example~\ref{exampleallderivativefinite}, we have
$$\Psi(\bar\theta)=\Psi_1(\theta_1)+\Psi_2\left(\frac{1-\alpha_1}{\alpha_1}\theta_1\right)-1.$$
Since $G_1(e,e|1)$ is finite, $\Psi_1(\theta_1)>0$.
Moreover, recall that $\Psi_1$ is strictly decreasing and satisfies $\Psi_1(0)=1$, hence $\Psi_1(\theta_1)<1$.
On the other hand, $\lim_{t\to \theta_2} \Psi_2(t)=0$, see for instance \cite[Example~9.18~(3)]{Woe00}, while $\Psi_2(0)=1$.
We deduce that $\Psi(\bar\theta)$ converges to $\Psi_1(\theta_1)>0$ as $\alpha_1$ tends to 1 and converges to $\Psi_1(\theta_1)-1<0$ as $\alpha_1$ tends to 0.
\end{Exl}

Proposition \ref{p:limit-harmonic} together with the examples above show the limitation in the strategy initiated in \cite{Woe21} to be applied to more general groups. This is because this strategy relies on finding the asymptotic behaviour of derivatives of the Green function at its singularities, which in turns uses the fact that $\lim_{r\to R}I^{(s)}(e,e|r)=\infty$.


\section{Essential minimality of ratio-limit boundary.} \label{s:main}

In this section we will show that when $\Gamma$ be a non-elementary relatively hyperbolic group, and $\mu$ is a finitely supported, admissible, aperiodic, symmetric and spectrally non-degenerate probability measure on $\Gamma$, then $\overline{\partial^m_{M,R}\Gamma}$ is the unique smallest closed $\Gamma$-invariant subspace of $\partial_{\rho}\Gamma$. In order to establish this we adapt a proof strategy of Woess from \cite{Woe21}, which was used for hyperbolic groups. Considering the radically different behavior of relatively hyperbolic groups exhibited in Sections \ref{s:spect-non-deg} and \ref{s:asymp}, our proofs require new ideas. We set
$$\Phi_r(x,y):= \frac{I^{(1)}(e,e|r)}{G(x,y|r)},$$ so that by Corollary \ref{c:ratio-limit-iterated-sums} we have
\begin{equation} \label{eq:conical_rl-vs-mart}
 \frac{\Phi_{R}(x,y)}{\Phi_{R}(e,y)}:= \lim_{r \rightarrow R} \frac{\Phi_r(x,y)}{\Phi_r(e,y)} = \frac{H(x,y)}{K_{R}(x,y)}.
\end{equation}

\begin{Prop} \label{p:conical-in-rlb}
Let $\Gamma$ be a non-elementary relatively hyperbolic group, and assume that $\mu$ is a finitely supported, admissible, aperiodic and symmetric probability measure on $\Gamma$. Suppose further that the random walk is spectrally non-degenerate, and that $\xi \in \partial_{M,R}\Gamma$ is a point such that $\pi(\xi)$ is conical point in $\partial_B(\Gamma;\Omega)$. Then, for every $x\in \Gamma$, as $y\rightarrow \xi$ in $\partial_{M,R}\Gamma$ we have
$$
\frac{\Phi_{R}(x,y)}{\Phi_{R}(e,y)} \underset{n \rightarrow \infty}{\longrightarrow} 1
$$  
\end{Prop}

\begin{proof}
First, as suggested in Figure \ref{fig:1}, we claim that if $x$ is fixed, then there exists a constant $C(x)$ such that if two relative geodesics $[e,y]$ and $[e,z]$ fellow-travel for a time $k$, then the relative geodesics $[x,z]$ and $[e,y]$ fellow-travel for a time $k-C(x)$.

To prove the claim, we let $w$ be the projection of $x$ on $[e,y]$. Then, we have by \cite[Lemma 4.16]{Dus22a} (see also \cite[Lemma 2.5]{Dus22b}) that any relative geodesic from $x$ to $y$ passes within a bounded distance in $\Gr(\Gamma)$ of $w$.
Since $d(w,x)$ is bounded by $|x|$, the claim follows, since by \cite[Lemma~4.14]{Dus22a} there exists $C(x)$ such that the two geodesics $[e,y]$ and $[x,y]$ $C(x)$-fellow travel, with $C(x)$ depending only on $d(w,x) \leq |x|$.

Following the computations in \cite[Subsection~6.1]{Dus22b} we decompose the sum over $z\in \Gamma$ as follows.
We write $[e,y]$ as $(e,y_1,...,y_n)$, where $n$ is the relative distance between $e$ and $y$ and we denote by $\Gamma_k$, $k\leq n$, the set of $z$ whose projection in $\Gr(\Gamma;\Omega)$ on $[e,y]$ is at $y_k$.
If there is more than one such projection, we choose the projection the one which is closest to $y$.
Then,
$$\Phi_r(e,y)=\sum_{k=0}^n\sum_{z\in \Gamma_k}\frac{G(e,z|r)G(z,y|r)}{G(e,y|r)}.$$
Now if $z\in \Gamma_k$, then $[e,z]$ and $[e,y]$ fellow travel for a time $k-1$ by \cite[Lemma~2.5]{Dus22b}.
Thus, $[x,z]$ and $[e,y]$ fellow travel for a time $k-1-C(x)$ by what we previously showed.
We then have
$$\big |\Phi_r(x,y)-\Phi_r(e,y)\big |\leq \sum_{k=0}^n\sum_{z\in \Gamma_k}\frac{G(e,z|r)G(z,y|r)}{G(e,y|r)}\left |\frac{G(x,z|r)G(e,y|r)}{G(x,y|r)G(e,z|r)}-1\right|.$$

\begin{figure}
\begin{tikzpicture}
\draw (0,0)--(3,3);
\draw (2,0)--(.7,.7);
\draw (2.6,2.6)--(1.7,3.2);
\draw (-.2,0) node{$e$};
\draw (3.2,3) node{$y$};
\draw (2.2,0) node{$x$};
\draw (1.5,3.2) node{$z$};
\end{tikzpicture}
\caption{} \label{fig:1}
\end{figure}

Assume now that $y$ converges to a point $\xi \in \partial_{M,R}\Gamma$ whose image $\pi(\xi)$ in Bowditch boundary is conical. Since $\pi$ is continuous, we know that $y$ converges to $\pi(\xi)$ in the relative metric. Hence, by strong relative Ancona inequality given in Proposition~\ref{strongrelativeAncona},
there exists $K$ depending only on $C(x)$ (i.e.\ on $x$) and $0<\alpha<1$ such that for $z\in \Gamma_k$,
$$\left |\frac{G(x,z|r)G(e,y|r)}{G(x,y|r)G(e,z|r)}-1\right|\leq K\alpha^k.$$
We thus find
\begin{equation}\label{equationboundIxy-Iey}
\big |\Phi_r(x,y)-\Phi_r(e,y)\big |\leq K\sum_{k=0}^n\alpha^k\sum_{z\in \Gamma_k}\frac{G(e,z|r)G(z,y|r)}{G(e,y|r)}.
\end{equation}

Let $H_k$ be the union of all parabolic subgroups containing $y_{k-1}^{-1}y_k$.
Then it is proved at the beginning of \cite[Subsection~6.1]{Dus22b} that
\begin{equation}\label{boundsumoverGammak}
\sum_{z\in \Gamma_k}\frac{G(e,z|r)G(z,y|r)}{G(e,y|r)}\asymp I^{(1)}(e,e|r)\sum_{w\in H_k}\frac{G(y_{k-1},y_{k-1}w|r)G(y_{k-1}w,y_k|r)}{G(y_{k-1},y_k|r)}.
\end{equation}
Only the upper bound for the sum on the left is written explicitly at the beginning of \cite[Subsection~6.1]{Dus22b}, but equation \eqref{boundsumoverGammak} is a direct consequence of \cite[Equation~(41), Proposition~6.3]{Dus22b}.
Still, let us give some details on the proof of the lower bound for the sake of completeness.
We write $y_{k-1}w$ for the projection of $z$ on $y_{k-1}H_k$.

Then as suggested in Figure \ref{fig:2}, we deduce from \cite[Lemma 4.16]{Dus22a} (see also \cite[Lemma 2.5]{Dus22b}) that a relative geodesic $[e,z]$ passes within a bounded distance of $y_{k-1}$, then of $y_{k-1}w$.
Similarly, a relative geodesic $[z,y]$ passes within a bounded distance of $y_{k-1}w$, then of $y_k$.
Finally, by definition, $[e,y]$ passes first at $y_{k-1}$ then at $y_k$.
Therefore, by weak relative Ancona inequalities,
$$G(e,z|r)\asymp G(e,y_{k-1}|r)G(y_{k-1},y_{k-1}w|r)G(y_{k-1}w,z|r),$$
$$G(z,y|r)\asymp G(z,y_{k-1}w|r)G(y_{k-1}w,y_{k}|r)G(y_{k},y|r)$$
and
$$G(e,y|r)\asymp G(e,y_{k-1}|r)G(y_{k-1},y_k|r)G(y_k,y|r).$$
Now, summing over $z\in \Gamma_k$ boils down to summing over $w\in H_k$ and over $z$ that projects on $y_{k-1}H_k$ at $y_{k-1}w$.
Moreover, for fixed $w$, the residual sum
$$\Sigma=\sum_{z}G(y_{k-1}w,z|r)G(z,y_{k-1}w|r)$$
is roughly asymptotic to $I^{(1)}(e,e|r)$. Indeed, by distinguishing elements $z$ according to their projection on $y_{k-1}H_k$ and using weak relative Ancona inequalities, we see that the ratio of $I^{(1)}(e,e|r)$ and $\Sigma$ is roughly given by $\sum_{w\in H_k}G(e,w|r)G(w,e|r)$, and the latter is uniformly finite (see the similar proof in \cite[(20)]{Dus22a}). All these inequalities together imply the desired lower bound in equation \eqref{boundsumoverGammak}.

Next, it will be convenient to use the following notations from \cite{Dus22b}.
Define the function $\Upsilon_r$ on $\Gamma$ as follows. For a point $y\in \Gamma$, let $[e,y]$ be some relative geodesic chosen according to an automaton (see \cite{Dus22b}), and let $y_1$ be the first point after $e$ in $[e,y]$. Then, we define
$$\Upsilon_r(y)=\sum_{w\in H_1}\frac{G(e,w|r)G(w,y_1|r)}{G(e,y_1|r)}.$$

\begin{figure}
\begin{tikzpicture}[scale=1.8]
\draw (-1.8,.5)--(-2,-.4);
\draw (2.2,.5)--(2,-.4);
\draw (-2,-.4)--(2,-.4);
\draw (-.7,-.8)--(-.7,-.4);
\draw[dotted] (-.7,-.4)--(-.7,0);
\draw (-.7,-.8)--(-2.3,-.8);
\draw (-.7,0)--(.8,-.1);
\draw (-.7,0)--(.5,.3);
\draw[dotted] (.8,-.1)--(.8,-.4);
\draw (.8,-.4)--(.8,-1);
\draw (.8,-1)--(1.6,-1);
\draw (.5,.3)--(.5,1.1);
\draw (.8,-.1)--(.5,.3);
\draw (-1,0)node{$y_{k-1}$};
\draw (-2.5,-.8) node{$e$};
\draw (1.1,-.15) node{$y_{k}$};
\draw (1.8,-1) node{$y$};
\draw (.5,1.3) node{$z$};
\draw (.9,.4) node{$y_{k-1}w$};
\end{tikzpicture}
\caption{} \label{fig:2}
\end{figure}

Denote by $T$ the left shift on relative geodesics starting at $e$. That is, if $\gamma=(e,\gamma_1,...,\gamma_n)$ is a relative geodesic, then $T\gamma=(e,\gamma_1^{-1}\gamma_2,...,\gamma_1^{-1}\gamma_n)$.
Then, combining~(\ref{equationboundIxy-Iey}) and~(\ref{boundsumoverGammak}), we have
\begin{equation}\label{boundusingupsilon}
\big |\Phi_r(x,y)-\Phi_r(e,y)\big |\lesssim I^{(1)}(e,e|r) K\sum_{k=0}^n\alpha^k \Upsilon_r\big (T^k[e,y]\big)
\end{equation}
whereas by \cite[Equation~(41), Proposition~6.3]{Dus22b} we have
\begin{equation}\label{equationupsilon}
\Phi_r(e,y)\asymp I^{(1)}(e,e|r) \sum_{k=0}^n \Upsilon_r\big (T^k[e,y]\big)
\end{equation}

Fix $\epsilon>0$. Then, there exists $k_0$ such that for $k\geq k_0$ we have $\frac{\alpha^k}{1-\alpha} \leq \epsilon$. Hence, by equations \eqref{boundusingupsilon} and \eqref{equationupsilon},
$$\big |\Phi_r(x,y)-\Phi_r(e,y)\big |\lesssim I^{(1)}(e,e|r) \sum_{k=0}^{k_0}\Upsilon_r\big (T^k[e,y]\big)+\epsilon \Phi_r(e,y)$$
Using again~(\ref{equationupsilon}), we rewrite this as
\begin{equation} \label{eq:last-ep-bound}
\left |\frac{\Phi_r(x,y)}{\Phi_r(e,y)}-1\right |\lesssim \frac{\sum_{k=0}^{k_0}\Upsilon_r\big (T^k[e,y]\big)}{\sum_{k=0}^{n}\Upsilon_r\big (T^k[e,y]\big)}+\epsilon.
\end{equation}
Finally, $k_0$ is fixed and if $n=|y|$ is big enough, for $k\leq k_0$ we have that $\Upsilon_r\big (T^k[e,y]\big)$ is bounded by a constant that only depends on $\xi$.
On the other hand, for every $k$,
$$\sum_{w\in H_k}\frac{G(y_{k-1},y_{k-1}w|r)G(y_{k-1}w,y_k|r)}{G(y_{k-1},y_k|r)}\gtrsim 1,$$ hence as $n$ tends to infinity,
$\sum_{k=0}^{n}\Upsilon_r\big (T^k[e,y]\big)$ tends to infinity. Thus, if we take $n = |y|$ large enough, we can arrange for the first summand in the right-hand-side of equation \eqref{eq:last-ep-bound} to be at most $\epsilon$, and get
$$\left |\frac{\Phi_r(x,y)}{\Phi_r(e,y)}-1\right |\lesssim 2\epsilon.$$
Since weak and strong relative Ancona inequalities are uniform in $r\leq R$, this last bound is also uniform in $r$, and concludes the proof.
\end{proof}

This allows us to show, when the random walk is spectrally non-deg\-enerate, that $\overline{\partial_{M,R}^m\Gamma}$ embeds $\Gamma$-equivariantly inside $\partial_{\rho}\Gamma$.

\begin{Cor} \label{cor:embedding}
Let $\Gamma$ be a non-elementary hyperbolic group relative to a finite collection of subgroups $\Omega$, and assume that $\mu$ is a finitely supported, admissible, aperiodic and symmetric probability measure on $\Gamma$. Suppose further that the random walk is spectrally non-degenerate. Then, there is a bi-Lipschitz $\Gamma$-equivariant map $\iota : \overline{\partial_{M,R}^m\Gamma} \rightarrow \partial_{\rho} \Gamma$ which is a homeomorphism on its image. Moreover, if we have $C_x=D_x$ for every $x\in \Gamma$ in equations \eqref{defmetricMartin} and \eqref{defmetricratio}, then $\iota$ is isometric.
\end{Cor}

\begin{proof}
Let $\pi : \partial_{M,R} \Gamma \rightarrow \partial_B(\Gamma;\Omega)$ be the canonical $\Gamma$-factor map (see discussion at the end of Section~\ref{s:rw-rh}). Suppose that $(y_n)$ is a sequence converging to $\xi$ in the $R$-Martin metric $d_{M,R}$, with $\pi(\xi)$ conical. Since $\pi$ extends to a $\Gamma$-factor from $\Delta_{M,R} \rightarrow \Gamma \cup \partial_B(\Gamma;\Omega)$, we get that $(y_n)$ converges to $\pi(\xi)$ in the relative metric as well. By Proposition \ref{p:conical-in-rlb} and equation \eqref{eq:conical_rl-vs-mart} we see that $y_n \rightarrow \iota(\xi)$ according to the metric $d_{\rho}$, for some point $\iota(\xi)$ in $\partial_{\rho} \Gamma$, so that $H(x,\iota(\xi)) = K_{R}(x,\xi)$.
By equations \eqref{defmetricMartin} and \eqref{defmetricratio}, since the metrics $d_{M,R}$ and $d_{\rho}$ are defined using $R$-Martin kernels and ratio-limit kernels respectively, we would get that $\iota$ is bi-Lipschitz map on the set of $\xi$ for which $\pi(\xi)$ is conical. If we take $C_x=D_x$ for every $x\in \Gamma$ in equations \eqref{defmetricMartin} and \eqref{defmetricratio} we actually get that $\iota$ is an isometry. By Corollary \ref{c:min-dense-con-dense} we know that the pre-images under $\pi$ of conical points are dense in $\overline{\partial_{M,R}^m\Gamma}$, so we may extend $\iota$ to a bi-Lipschitz map from $ \overline{\partial_{M,R}^m\Gamma}$ to $ \partial_{\rho} \Gamma$, and if $\iota$ is isometric on conical points, it would extend to an isometry from $ \overline{\partial_{M,R}^m\Gamma}$ to $ \partial_{\rho} \Gamma$.

The above relation $H(x,\iota(\xi)) = K_R(x,\xi)$ together with $K_R(x,g\xi)K_R(g^{-1},\xi)=K_R(g^{-1}x,\xi)$ for the Martin kernel and
$H(x,g\xi)H(g^{-1},\xi)=H(g^{-1}x,\xi)$ for the ratio limit kernel show that $\iota$ is 
$\Gamma$-equivariant on the set of $\xi$ such that $\pi(\xi)$ is conical.
Since $\Gamma$ acts on both boundaries via homeomorphisms, if $\xi_n$ converges to $\xi$ and if $g\in \Gamma$, then $g\xi_n$ converges to $g\xi$.
Now, the set of $\xi$ such that $\pi(\xi)$ is conical is dense in $\overline{\partial_{M,R}^m\Gamma}$.
Letting $\xi_n$ with $\pi(\xi_n)$ conical tend to $\xi\in \overline{\partial_{M,R}^m\Gamma}$, we see that
$$\iota (g\xi)=\lim_n \iota (g\xi_n)=\lim_n g\iota (\xi_n)=g\iota (\xi).$$
Therefore, the extension of $\iota$ to $\overline{\partial_{M,R}^m\Gamma}$ is still $\Gamma$-equivariant.

In any case, since $\iota$ is bi-Lipschitz and $\overline{\partial_{M,R}^m\Gamma}$ is compact, we deduce that $\iota$ is a homeomorphism on its image.
\end{proof}

\begin{Cor} \label{cor:reduced}
Let $\Gamma$ be a non-elementary hyperbolic group relative to a finite collection of subgroups $\Omega$, and assume that $\mu$ is aperiodic, symmetric probability measure with finite support generating $\Gamma$. Suppose further that the random walk is spectrally non-degenerate. Then $\theta |_{\partial_{\rho} \Gamma} : \partial_{\rho} \Gamma \rightarrow \partial_{\rho}^r \Gamma$ is a $\Gamma$-equivariant homeomorphism.
\end{Cor}

\begin{proof}
Since $R_{\mu}$ is normal, it acts trivially on $\Gamma/R_{\mu}$, and therefore also on $\Delta_{\rho}^r \Gamma$. By \cite[Lemma 6.2(ii)]{Woe21} we know that $\theta|_{\partial_{\rho} \Gamma} : \partial_{\rho} \Gamma \rightarrow \Delta_{\rho}^r \Gamma$ is injective, so that $R_{\mu}$ acts trivially also on $\partial_{\rho} \Gamma$. By Corollary \ref{cor:embedding} we get that $R_{\mu}$ acts trivially on $\overline{\partial^m_{M,R} \Gamma}$. Since the pre-image of every point in $\partial_B(\Gamma;\Omega)$ contains a minimal point (see end of Section \ref{s:rw-rh})
we see that the restriction of $\pi$ to $\overline{\partial^m_{M,R} \Gamma}$ is still onto $\partial_B(\Gamma;\Omega)$, and we get that $R_{\mu}$ acts trivially on $\partial_B(\Gamma;\Omega)$. Now, since the action of $\Gamma$ on the space of distinct triples is properly discontinuous, it follows that the kernel of $\Gamma \curvearrowright \partial_B(\Gamma;\Omega)$ is finite, so that $R_{\mu}$ is also finite.

Now let $\xi \in \partial_{\rho} \Gamma$ be a boundary element, and $(g_n)$ a sequence in $\Gamma$ such that $g_n \rightarrow \xi$ in $\Delta_{\rho} \Gamma$. As $R_{\mu}$ is finite, for any finite subset $F\subset \Gamma/R_{\mu}$ we must have that $g_nR_{\mu}$ is eventually not in $F$. Hence, we must have that $g_n R_{\mu} = \theta(g_n) \rightarrow \theta(\xi)$ is not in $F$, so that $\theta(\xi) \in \partial_{\rho}^r \Gamma$. Since $\theta$ is surjective, $\theta |_{\partial_{\rho} \Gamma}$ must be onto $\partial_{\rho}^r \Gamma$, and is therefore a homeomorphism.
\end{proof}

Combining Corollary \ref{cor:embedding} with results from Section \ref{s:min-strong-prox}, we obtain one of the main results of the paper, showing that $\partial_{\rho} \Gamma$ is essentially minimal. That is, $\overline{\partial^m_{M,R}\Gamma}$ is the unique smallest closed $\Gamma$-invariant subspace of $\partial_{\rho} \Gamma$.

\begin{Thm} \label{t:essent-min}
Let $\Gamma$ be a non-elementary hyperbolic group relative to a finite collection of subgroups $\Omega$, and assume that $\mu$ is aperiodic, symmetric probability measure with finite support generating $\Gamma$. Suppose further that the random walk is spectrally non-degenerate. Then $\Gamma \curvearrowright \partial_{\rho} \Gamma$ is essentially minimal, and its unique smallest closed $\Gamma$-invariant subspace is $\Gamma \curvearrowright \overline{\partial^m_{M,R}\Gamma}$.
\end{Thm}

\begin{proof}
Let $\iota : \overline{\partial_{M,R}^m\Gamma} \rightarrow \partial_{\rho} \Gamma$ be the map constructed in Corollary~\ref{cor:embedding}, and  $\eta \in \partial_{\rho} \Gamma \setminus \iota(\overline{\partial_{M,R}^m\Gamma})$. By the discussion on $r$-harmonic functions in Section~\ref{s:rw-rh} we know that $\varphi : \Prob(\partial_{M,R}\Gamma) \rightarrow \mathcal{H}^+(P,R^{-1})$ given by $\varphi(\nu) = u_{\nu}$ is a surjective continuous map. 

Let us denote by $\nu^{\eta}$ the unique representing measure for $x\mapsto H(x,\eta)$ satisfying $\nu^{\eta}(\partial_{M,R}^m \Gamma) = 1$. For every $g\in \Gamma$ we have that,
\begin{align*}
H(g^{-1}, \eta) H(x,g \eta) &= H(g^{-1}x,\eta) = \int_{\partial_{M,R} \Gamma} K_{R}(g^{-1}x,\xi) d\nu^{\eta}(\xi)\\
&=\int_{\partial_{M,R} \Gamma} K_{R}(x,g\xi) K_{R}(g^{-1},\xi)d\nu^{\eta}(\xi) \\
&= \int_{\partial_{M,R} \Gamma} K_{R}(x,\xi) K_{R}(g^{-1},g^{-1}\xi)d\nu^{\eta}(g^{-1}\xi).
\end{align*}

Hence, by uniqueness of representing measures with full measures on $\partial_{M,R}^m \Gamma$, and as $\partial_{M,R}^m \Gamma$ is a $\Gamma$-invariant set, we get the Radon-Nikodym derivative
$$
\frac{d g_*\nu^{\eta}}{d\nu^{g\eta}}(\xi) = \frac{K_{R}(g^{-1},g^{-1} \xi)}{H(g^{-1},\eta)}.
$$
By Proposition \ref{p:min-strong-prox} we know that $\Gamma\curvearrowright \overline{\partial_{M,R}^m \Gamma}$ is minimal and strongly proximal, and contains all minimal points. Since the support of $\nu^{\eta}$ is contained in $\overline{\partial_{M,R}^m \Gamma}$, by strong proximality of $\Gamma\curvearrowright \overline{\partial_{M,R}^m \Gamma}$ there exists a net $(g_{\alpha})$ in $\Gamma$ so that $(g_{\alpha})_* \nu^{\eta}$ converges to some dirac measure $\delta_{\xi}$ for a minimal point $\xi \in \partial_{M,R}^m \Gamma$. Thus, since $\nu^{g\eta}$ are all probability measures, it also follows that $\nu^{g_{\alpha}\eta}$ converges to $\delta_{\xi}$. Since $\varphi : \Prob(\partial_{M,R}\Gamma) \rightarrow \mathcal{H}^+(P,R^{-1})$ is continuous, we get for all $x\in \Gamma$ that $H(x,g_{\alpha}\eta) \rightarrow K_{R}(x,\xi)$. But now, if without loss of generality (up to taking a subnet) $g_{\alpha}\eta$ converges to some cluster point $\eta' \in \partial_{\rho}\Gamma$, then $H(x,\eta') = K_{R}(x,\xi)$ for all $x\in \Gamma$, so that $\iota(\eta') = \xi$ is a minimal point. Since $\Gamma \curvearrowright \iota(\overline{\partial_{M,R}^m \Gamma})$ is a minimal action, we see that $\overline{\partial_{M,R}^m \Gamma}$ must be the unique smallest closed $\Gamma$-invariant subspace of $\partial_{\rho} \Gamma$.
\end{proof}


\section{Distinct ratio-limit boundaries on $\mathbb{Z}^5 * \mathbb{Z}$}

In this section, we provide an example of a group endowed with two probability measures such that the resulting ratio limit boundaries are not equivariantly homeomorphic. More specifically, we take $\Gamma= \Gamma_1 * \Gamma_2$ where $\Gamma_1 = \mathbb{Z}^5$ and $\Gamma_2 = \mathbb{Z}$, and $\mu = \alpha_1 \mu_1 + \alpha_2 \mu_2$ will be an finitely supported, admissible, adapted and symmetric probability measures. By Example \ref{ex:limitprobmeasures} we may fix $\alpha_1$ large enough so that the random walk is convergent. We introduce the function $\Phi_r^{(2)}$ defined by
\begin{equation}\label{defphi2}
\Phi_r^{(2)}(x,y)=\frac{I^{(2)}(x,y|r)}{G(x,y|r)}=\frac{1}{G(x,y|r)}\sum_{z_1,z_2\in \Gamma}G(x,z_1|r)G(z_1,z_2|r)G(z_2,y|r).
\end{equation}
By Corollary~\ref{c:ratio-limit-iterated-sums}~(2) we have that,
$$\frac{\Phi_R^{(2)}(x,y)}{\Phi_R^{(2)}(e,y)}:= \lim_{r\to R}\frac{\Phi_r^{(2)}(x,y)}{\Phi_r^{(2)}(e,y)} = \frac{H(x,y)}{K(x,y|R)}.$$

Our goal is to prove the following analogous of Proposition~\ref{p:conical-in-rlb} in this special case. The proof follows the same line as the proof of Proposition~\ref{p:conical-in-rlb}. The main difference is that we need to deal with $\Phi_r^{(2)}$ instead of $\Phi_r$.
However, in our specific context, most geometric arguments are replaced with combinatorial ones that use the free product structure. We will use the more convenient terminology of free products rather than the terminology of relatively hyperbolic groups, which is better suited for our proof.

\begin{Prop}\label{p:conicalconvergentcase}
Let $\Gamma = \Gamma_1 * \Gamma_2$ where $\Gamma_1 = \mathbb{Z}^5$ and $\Gamma_2 = \mathbb{Z}$. Let $\mu_1,\mu_2$ be finitely supported, admissible and symmetric probability measures on $\Gamma_1,\Gamma_2$ respectively, such that the adapted probability measure $\mu = \alpha_1 \mu_1 + \alpha_2 \mu_2$ is convergent and spectrally degenerate along $\Gamma_1$. Let $\xi \in \partial_{M,R}\Gamma$ be a point such that $\pi(\xi)$ is conical in $\partial_B(\Gamma;\Omega)$. Then for every $x\in \Gamma$, as $y\rightarrow \xi$ in $\partial_{M,R}\Gamma$ we have
$$
\frac{H(x,y)}{K_R(x,y)}=\frac{\Phi_R^{(2)}(x,y)}{\Phi_R^{(2)}(e,y)} \underset{y \rightarrow \xi}{\longrightarrow} 1.
$$
\end{Prop}

\begin{proof}
Before we begin our proof, we describe our overall strategy. Our first goal is to prove that
$$\Phi_r^{(2)}(e,y)-\Phi_r^{(2)}(x,y)\sim \Tilde{\Phi}_r^{(2)}(e,y)-\Tilde{\Phi}_r^{(2)}(x,y),$$ where $\Tilde{\Phi}_r^{(2)}$ is a twisted Birkhoff sum that replaces the expression $\sum \Upsilon_r(T^k[e,y])$ in the proof of Proposition~\ref{p:conical-in-rlb}.
This twisted Birkhoff sum involves the second derivatives of the induced Green functions on free factors.
Unlike $\sum \Upsilon_r(T^k[e,y])$, it tends to infinity as $r$ goes to $R$, even for fixed $y$ and we need to understand its asymptotic behavior in $r$.
To do so, we use classical results about the expansion of the Green function at its spectral radius in $\mathbb Z^d$ and show that $\Tilde{\Phi}_r^{(2)}(e,y)$ behaves like  $(1-\zeta_1(r))^{-1/2}$.
Finally, as $y$ tends to a conical limit point, we prove that for fixed $x$, the terms in $\Tilde{\Phi}_r^{(2)}(e,y)-\Tilde{\Phi}_r^{(2)}(x,y)$ simplify each other so $(1-\zeta_1(r))^{-1/2}\big(\Tilde{\Phi}_r^{(2)}(e,y)-\Tilde{\Phi}_r^{(2)}(x,y)\big)$ remains bounded, while $(1-\zeta_1(r))^{-1/2}\Tilde{\Phi}_r^{(2)}(e,y)$ goes to infinity.
We conclude that $\lim_{r\to R}\frac{\Tilde{\Phi}_r^{(2)}(e,y)-\Tilde{\Phi}_r^{(2)}(x,y)}{ \Tilde{\Phi}_r^{(2)}(e,y)}$ converges to 0 as $y$ converges to a conical limit point, so that $H(x,y)/K_R(x,y)$ converges to 1.

We first compute $\Phi^{(2)}_r(e,y)$ as defined in~(\ref{defphi2}).
We write $y$ in its normal form $y=s_1...s_n$, where $s_i$ and $s_{i+1}$ for $i=1,...,n-1$ belong to different parabolic subgroups (either $\mathbb{Z}^5$ or $\mathbb{Z}$), and we write $y_j=s_1...s_j$. 

\begin{figure}
\begin{center}
\begin{tikzpicture}[scale=2.2]
\draw (-1.8,.5)--(-2,-.4);
\draw (2.2,.5)--(2,-.4);
\draw (-2,-.4)--(2,-.4);
\draw (-.2,-1.3)--(-.2,-.4);
\draw[dotted] (-.2,-.4)--(-.2,0);
\draw (-.2,0)--(.4,-.2);

\draw (-.2,0)--(.2,.4);

\draw[dotted] (.4,-.2)--(.4,-.4);
\draw (.4,-.4)--(.4,-1);
\draw (.4,-.2)--(.2,.4);
\draw (.2,.4)--(.2,1);
\draw[dotted] (.2,1)--(.2,1.4);
\draw (-2.3,1.9)--(-2.5,1);
\draw (.7,1.9)--(.5,1);
\draw (-2.5,1)--(.5,1);
\draw (.2,1.4)--(-1.6,1.7);
\draw (-1.6,1.7)--(-.5,1.3);
\draw (-.5,1.3)--(.2,1.4);
\draw (-1.6,1.7)--(-1.6,2.5);
\draw (-.5,1.3)--(-.5,2.3);
\draw (-.5,-1.3) node{$e$};
\draw (-.4,.05) node{$y_{j-1}$};
\draw (.1,.45) node{$y_{j}$};
\draw (.7,-.15) node{$y_{j-1}w_1$};
\draw (.6,-1.1) node{$z_1$};
\draw (-1.7,2.5) node{$z_2$};
\draw (-1.95, 1.7) node{$y_{k-1}w_2$};
\draw (.4,1.4) node{$y_{k-1}$};
\draw (-.45,1.2) node{$y_k$};
\draw (-.4,2.3) node{$y$};
\end{tikzpicture}
\end{center}
\caption{} \label{fig:3}
\end{figure}

Fix two points $z_1,z_2 \in \Gamma$, and let $y_{j-1}$ be the largest common prefix of $y$ and $z_1$. Similarly, let $y_{k-1}$ be the largest common prefix of $y$ and $z_2$.
We also write $w_1$ and $w_2$ for the next letter (either after $y_{j-1}$ or $y_{k-1}$) in the normal form of $z_1$ and $z_2$.
In other words, $z_1$ and $z_2$ can be written as
$$z_1=y_{j-1}w_1z'_1$$
$$z_2=y_{k-1}w_2z'_2,$$
with $w_i\in \mathbb Z^5$ or $w_i\in \mathbb Z$ and no prefix of $z_1',z_2'$ in the same subgroup that $w_1$ and $w_2$ belong to.

Any path from $e$ to $z_1$ necessarily has to pass through $y_{j-1}$ and $y_{j-1}w_1$.
By \cite[Equation (3.3)]{Woe86} we may decomposing such a path to get
$$\frac{G(e,z_1|r)}{G(e,e|r)}=\frac{G(e,y_{j-1}|r)}{G(e,e|r)}\frac{G(y_{j-1},y_{j-1}w_1|r)}{G(e,e|r)}\frac{G(y_{j-1}w_1,z_1|r)}{G(e,e|r)}.$$
We rewrite this as

\begin{equation} \label{eq:etoz1}
   G(e,e|r)^2G(e,z_1|r)=G(e,y_{j-1}|r)G(e,w_1|r)G(e,z'_1|r). 
\end{equation}

Similarly, any path from $z_2$ to $y$ has to pass through $y_{k-1}w_2$ and $y_k$, hence
$$\frac{G(z_2,y|r)}{G(e,e|r)}=\frac{G(z_2,y_{k-1}w_2|r)}{G(e,e|r)}\frac{G(y_{k-1}w_2,y_{k}|r)}{G(e,e|r)}\frac{G(y_{k}w_1,y|r)}{G(e,e|r)},$$
We rewrite this as
\begin{equation} \label{eq:z2toy}
G(e,e|r)^2G(z_2,y|r)=G(z_2',e|r)G(w_2,s_k|r)G(y_k,y|r).
\end{equation}

We have two cases to consider, either $j=k$ or $j\neq k$.
Our goal is ultimately to prove that the dominant term in $\Phi^{(2)}(e,y)$ is the sub-sum over $z_1$ and $z_2$ such that $j=k$.
We will prove that the sub-sum over $z_1$ and $z_2$ such that $j\neq k$ converges to a finite limit as $r$ tends to $R$, while the sub-sum over $z_1$ and $z_2$ such that $j=k$ diverges as $r$ tends to $R$. 

We first assume that $j< k$ and consider Figure \ref{fig:3}. Then, using \cite[Equation (3.3)]{Woe86} again, we similarly get that for any path from $z_1$ to $z_2$ has to pass through $y_{j-1}w_1$, $y_{j}$, $y_{k-1}$ and $y_{k-1}w_2$ so we have

\begin{equation} \label{eq:z1z2}
G(e,e|r)^4G(z_1,z_2|r)=G(z_1',e|r)G(w_1,s_j|r)G(y_{j},y_{k-1}|r)G(e,w_2|r)G(e,z_2'|r).
\end{equation}
In the same manner, any path from $e$ to $y$ has to pass through $y_{j-1}$, $y_j$, $y_{k-1}$ and $y_k$, so
\begin{equation} \label{eq:etoy}
G(e,e|r)^4G(e,y|r)=G(e,y_{j-1}|r)G(e,s_j|r)G(y_{j},y_{k-1}|r)G(e,s_k|r)G(y_k,y|r).
\end{equation}
By multiplying \eqref{eq:etoz1}, \eqref{eq:z2toy}, \eqref{eq:z1z2} and dividing by \eqref{eq:etoy} we get
\begin{equation}\label{sumjsmallerk}
\begin{split}
G(e,e|r)^4\frac{G(e,z_1|r)G(z_1,z_2|r)G(z_2,y|r)}{G(e,y|r)}=& \ \frac{G(e,w_1|r)G(w_1,s_j|r)}{G(e,s_j|r)}\frac{G(e,w_2|r)G(w_2,s_k|r)}{G(e,s_k|r)} \cdot \\
&\hspace{.1cm}G(e,z'_1|r)G(z'_1,e|r)G(e,z'_2|r)G(z'_2,e|r).
\end{split}
\end{equation}

We now sum~(\ref{sumjsmallerk}) over all $z_1$ and $z_2$ such that $j<k$.
This is equivalent to summing over $k$ and then over $j<k$, $w_1$ in the same free factor as $s_j$, $z'_1$ that start with a letter in the other free factor than $s_j$, $w_2$ in the same free factor as $s_k$, and $z'_2$ that start with a letter in the other free factor than $s_k$.
For $\ell = 1,...,n$ we set $\Gamma_{i_{\ell}}$ the free factor to which such that $s_{\ell}$ belongs to, and we write $\Gamma_{i_{\ell}}^{\perp}$ for the set of elements that start with a letter not in $\Gamma_{i_{\ell}}$.
By changing summation as above we obtain that the following expression is equal to the one obtained by summing equation \eqref{sumjsmallerk} over $z_1,z_2$, $j<k$
\begin{align*}
\sum_{k\leq n}\sum_{j<k}&\sum_{w_1\in \Gamma_{i_j}}\frac{G(e,w_1|r)G(w_1,s_j|r)}{G(e,s_j|r)}\sum_{w_2\in \Gamma_{i_k}}\frac{G(e,w_2|r)G(w_2,s_k|r)}{G(e,s_k|r)} \cdot \\
&\hspace{.5cm}\sum_{z'_1\in \Gamma_{i_j}^\perp}G(e,z_1'|r)G(z'_1,e|r)\sum_{z'_2\in \Gamma_{i_k}^\perp}G(e,z_2'|r)G(z'_2,e|r).
\end{align*}

\begin{figure}
\begin{center}
\begin{tikzpicture}[scale=2.2]
\draw (-1.8,.5)--(-2,-.4);
\draw (2.2,.5)--(2,-.4);
\draw (-2,-.4)--(2,-.4);
\draw (-.2,-1.3)--(-.2,-.4);
\draw[dotted] (-.2,-.4)--(-.2,0);
\draw (-.2,0)--(.4,-.2);

\draw (-.2,0)--(.2,.4);

\draw[dotted] (.4,-.2)--(.4,-.4);
\draw (.4,-.4)--(.4,-1);
\draw (.4,-.2)--(.2,.4);
\draw (.2,.4)--(.2,1);
\draw[dotted] (.2,1)--(.2,1.4);
\draw (-2.3,1.9)--(-2.5,1);
\draw (.7,1.9)--(.5,1);
\draw (-2.5,1)--(.5,1);
\draw (.2,1.4)--(-1.6,1.7);
\draw (-1.6,1.7)--(-.5,1.3);
\draw (-.5,1.3)--(.2,1.4);
\draw (-1.6,1.7)--(-1.6,2.5);
\draw (-.5,1.3)--(-.5,2.3);
\draw (-.5,-1.3) node{$e$};
\draw (-.4,.05) node{$y_{k-1}$};
\draw (.1,.45) node{$y_{k}$};
\draw (.7,-.15) node{$y_{k-1}w_2$};
\draw (.6,-1.1) node{$z_2$};
\draw (-1.7,2.5) node{$z_1$};
\draw (-1.95, 1.7) node{$y_{j-1}w_1$};
\draw (.4,1.4) node{$y_{j-1}$};
\draw (-.45,1.2) node{$y_j$};
\draw (-.4,2.3) node{$y$};
\end{tikzpicture}
\end{center}
\caption{} \label{fig:4}
\end{figure}

Since the random walk is convergent we know that $I^{(1)}(e,e|R)$ is finite. Consequently, the two last inner sums in equation \eqref{sumjsmallerk} over $z_1'$ and $z_2'$ converge to a finite limit as $r$ tends to $R$.
Next, by \cite[Proposition~6.3]{DG21}, the sums over $w_1$ and $w_2$ also converge to a finite limit as $r$ tends to $R$.
We conclude that the sum over all $z_1$ and $z_2$ such that $j<k$ converges to a finite limit as $r$ tends to $R$.

Second, if we assume that $j>k$ a similar analysis (see Figure \ref{fig:4}) yields

\begin{equation}\label{sumjbiggerk}
\begin{split}
G(e,e|r)^8\frac{G(e,z_1|r)G(z_1,z_2|r)G(z_2,y|r)}{G(e,y|r)}=& \ G(e,w_1|r)G(w_1,e|r)G(s_k,w_2|r)G(w_2,s_k|r)\\
&\hspace{.1cm}G(e,z'_1|r)G(z'_1,e|r)G(e,z'_2|r)G(z'_2,e|r).\\
&\hspace{.1cm}G(y_k,y_{j-1}|r)G(y_{j-1},y_k|r)
\end{split}
\end{equation}
Summing the left hand side of equation \eqref{sumjbiggerk} over $z_1$ and $z_2$ such that $j > k$ is equivalent to summing the right hand side over $j>k$, $w_1$, $w_2$, $z'_1$ and $z_2'$ as before, and we again find that this quantity converges to a finite limit as $r$ tends to $R$.

Thus, we are left with analyzing the case where $j=k$.
\begin{figure}
\begin{center}
\begin{tikzpicture}[scale=1.8]
\draw (-1.8,.5)--(-2,-.4);
\draw (2.2,.5)--(2,-.4);
\draw (-2,-.4)--(2,-.4);
\draw (-.7,-.8)--(-.7,-.4);
\draw[dotted] (-.7,-.4)--(-.7,0);
\draw (-.7,-.8)--(-2.3,-.8);
\draw (-.7,0)--(.8,-.1);
\draw (-.7,0)--(-.2,.3);
\draw (-.2,.3)--(.5,.3);
\draw[dotted] (.8,-.1)--(.8,-.4);
\draw (.8,-.4)--(.8,-1);
\draw (.8,-1)--(1.6,-1);
\draw (.5,.3)--(.5,1.1);
\draw (.8,-.1)--(.5,.3);
\draw (-.2,.3)--(-.2,1.1);
\draw (-1,0)node{$y_{k-1}$};
\draw (-2.5,-.8) node{$e$};
\draw (1.1,-.15) node{$y_{k}$};
\draw (1.8,-1) node{$y$};
\draw (.5,1.3) node{$z_2$};
\draw (-.2,1.3) node{$z_1$};
\draw (-.55,.4) node{$y_{k-1}w_1$};
\draw (.9,.4) node{$y_{k-1}w_2$};
\end{tikzpicture}
\end{center}
\caption{} \label{fig:5}
\end{figure}
This time, similarly to before, following paths in Figure \ref{fig:5} we obtain
\begin{equation}\label{sumjequalk}
\begin{split}
G(e,e|r)^4\frac{G(e,z_1|r)G(z_1,z_2|r)G(z_2,y|r)}{G(e,y|r)}=& \ \frac{G(e,w_1|r)G(w_1,w_2|r)G(w_2,s_k|r)}{G(e,s_k|r)} \cdot \\
&\hspace{.1cm}G(e,z'_1|r)G(z'_1,e|r)G(e,z'_2|r)G(z'_2,e|r).
\end{split}
\end{equation}
As before, summing over $z_1$ and $z_2$ boils down to summing over $k$, $w_1,w_2\in \Gamma_{i_k}$ and $z_1,z_2\in \Gamma_{i_k}^\perp$.
Hence, we get that the expression obtained by summing equation \eqref{sumjequalk} is equal to
\begin{align*}
\sum_{k\leq n}&\sum_{w_1,w_2\in \Gamma_{i_k}}\frac{G(e,w_1|r)G(w_1,w_2|r)G(w_2,s_k|r)}{G(e,s_k|r)}\\
&\hspace{.5cm}\sum_{z'_1,z'_2\in \Gamma_{i_k}^\perp}G(e,z_1'|r)G(z'_1,e|r)G(e,z_2'|r)G(z'_2,e|r).
\end{align*}
We use again that the random walk is convergent to deduce that the inner sums over $z_1'$ and $z_2'$ converge to a finite limit as $r$ tends to $R$.

Thus, we need only focus on the sum over $w_1$ and $w_2$ given by
\begin{align*}
\sum_{k\leq n}&\sum_{w_1,w_2\in \Gamma_{i_k}}\frac{G(e,w_1|r)G(w_1,w_2|r)G(w_2,s_k|r)}{G(e,s_k|r)}.
\end{align*}
By~(\ref{formulaWoess}), we have
$$\sum_{w_1,w_2\in \Gamma_{i_k}}\frac{G(e,w_1|r)G(w_1,w_2|r)G(w_2,s_k|r)}{G(e,s_k|r)}=\frac{G(e,e|r)^2}{G_{i_k}(e,e|\zeta_{i_k}(r))^2}\frac{I_{i_k}^{(2)}(e,s_k|\zeta_{i_k}(r))}{G_{i_k}(e,s_k|\zeta_{i_k}(r))}.$$
We set
$$\alpha_k(r):=\frac{G(e,e|r)^2}{G_{i_k}(e,e|\zeta_{i_k}(r))^2}\frac{1}{G(e,e|r)^4}\sum_{z'_1,z'_2\in \Gamma_{i_k}^\perp}G(e,z_1'|r)G(z'_1,e|r)G(e,z_2'|r)G(z'_2,e|r)$$ and
$$\Tilde{\Phi}_r^{(2)}(e,y):=\sum_{k=1}^n\alpha_k(r)\frac{I_{i_k}^{(2)}(e,s_k|\zeta_{i_k}(r))}{G_{i_k}(e,s_k|\zeta_{i_k}(r))},$$
Thus, by summing $z_1$ and $z_2$ in the left hand side of equation \eqref{sumjequalk} in the case where $j=k$ together with the convergence of appropriate summations of the expression in equations \eqref{sumjsmallerk} for $j<k$ and \eqref{sumjbiggerk} for $j>k$ we find that
$$\Phi_r^{(2)}(e,y)=\Tilde{\Phi}_r^{(2)}(e,y)+O(1), r\to R,$$
and the coefficients $\alpha_k(r)$ converge to some finite positive limit $\alpha_k(R)$ as $r$ tends to $R$.
Now, the random walk is spectrally degenerate along $\Gamma_1=\mathbb Z^5$, hence $I^{(2)}_1(e,g|\zeta_1(R))$ is infinite for $g\in \Gamma_1$.
Therefore, as long as the normal form of $y$ has one letter in $\Gamma_1$ (which necessarily happens as soon as it has more than one letter), we get that $\Tilde{\Phi}_r^{(2)}(e,y)$ goes to infinity as $r$ tends to $R$. Hence, we see that for fixed $y\in \Gamma$, as $r$ tends to $R$ we have  that
\begin{equation}\label{phiandtildephi}
    \Phi_r^{(2)}(e,y)\sim \Tilde{\Phi}_r^{(2)}(e,y),
\end{equation}
which may not be uniform in $y$.

We now consider $\Phi_r^{(2)}(x,y)$.
By writing $x^{-1}y$ in its normal form $x^{-1}y=t_1...t_m$ with $t_i\in \mathbb Z^5$ or $t_i\in \mathbb Z$, we may reduce to the case where $x=e$. Hence, similarly we get that
$$\tilde{\Phi}_r^{(2)}(x,y)=\sum_{k=1}^m\beta_k(r)\frac{I_{i_k}^{(2)}(e,t_k|\zeta_{i_k}(r))}{G_{i_k}(e,t_k|\zeta_{i_k}(r))}$$
with coefficients $\beta_k$ such that as $r\to R$, $\beta_k(r)$ converges to a finite limit and
$$\Phi_r^{(2)}(x,y)\sim \Tilde{\Phi}_r^{(2)}(x,y).$$

We now let $y_{l-1}$ be the largest common prefix of $x$ and $y$ and again write $x^{-1}y=t_1...t_m$ in its normal form.
\begin{figure}
\begin{center}
\begin{tikzpicture}[scale=1.8]
\draw (-1.8,.5)--(-2,-.4);
\draw (2.2,.5)--(2,-.4);
\draw (-2,-.4)--(2,-.4);
\draw (-.7,-.8)--(-.7,-.4);
\draw[dotted] (-.7,-.4)--(-.7,0);
\draw (-.7,-.8)--(-2.3,-.8);
\draw (-.7,0)--(.8,-.1);
\draw (-.7,0)--(.5,.3);
\draw[dotted] (.8,-.1)--(.8,-.4);
\draw (.8,-.4)--(.8,-1);
\draw (.8,-1)--(1.6,-1);
\draw (.5,.3)--(.5,1.1);
\draw (.8,-.1)--(.5,.3);
\draw (-1,0)node{$y_{l-1}$};
\draw (-2.5,-.8) node{$e$};
\draw (1.1,-.15) node{$y_l$};
\draw (1.8,-1) node{$y$};
\draw (.5,1.3) node{$x$};
\end{tikzpicture}
\end{center}
\caption{} \label{fig:6}
\end{figure}
In particular, $s_k=t_{k+m-n}$ and $\alpha_k=\beta_{k+m-n}$ for $k\geq l+1$.
Consequently, the terms in $\Tilde{\Phi}_r^{(2)}(e,y)$ and $\Tilde{\Phi}_r^{(2)}(x,y)$ eventually simplify each others and so as $r \to R$ we get
$$\Phi_r^{(2)}(x,y)-\Phi_r^{(2)}(e,y)=\Tilde{\Phi}_r^{(2)}(x,y_l)-\Tilde{\Phi}_r^{(2)}(e,y_l)+O(1).$$
Indeed, this is exactly because setting $k_0=m-n$, we have $s_k=t_{k+k_0}$ and $\alpha_k=\beta_{k+k_0}$ for $k\geq l+1$, so that
$$
\alpha_k \frac{I_{i_k}^{(2)}(e,s_k|\zeta_{i_k}(r))}{G_{i_k}(e,s_k|\zeta_{i_k}(r))}=\beta_{k+k_0} \frac{I_{i_{k+k_0}}^{(2)}(e,t_{k+k_0}|\zeta_{i_{k+k_0}}(r))}{G_{i_{k+k_0}}(e,t_{k+k_0}|\zeta_{i_{k+k_0}}(r))}.
$$ 
Consequently, the terms after $k=l+1$ cancel each other when taking the difference.

Since we have,
$$\frac{H(x,y)}{K_R(x,y)}-1=\lim_{r\to R}\frac{\Phi_r^{(2)}(x,y)-\Phi_r^{(2)}(e,y)}{\Phi_r^{(2)}(x,y)}=\lim_{r\to R}\frac{\Tilde{\Phi}_r^{(2)}(x,y_l)-\Tilde{\Phi}_r^{(2)}(e,y_l)}{\Tilde \Phi_r^{(2)}(e,y)},$$
in order to conclude the proof we need only show that
\begin{equation}\label{e:conclusionproof7.1}
\lim_{y\to \xi}\lim_{r\to R}\frac{\Tilde{\Phi}_r^{(2)}(x,y_l)-\Tilde{\Phi}_r^{(2)}(e,y_l)}{\Tilde \Phi_r^{(2)}(e,y)}=0.
\end{equation}

It follows from \cite[Theorem~13.10]{Woe00} that for every $g\in \Gamma_1=\mathbb Z^5$,
\begin{equation} \label{eq:z^dllt}
\bigg |C_0n^{5/2}\mu_{1}^{(n)}(g)-\mathrm{e}^{-\frac{\Sigma(g)}{2n}}\bigg|\underset{n\to \infty}{\longrightarrow}0,
\end{equation}
where $C_0>0$ is a constant and $\Sigma$ is a quadratic form associated with the random walk on $\mathbb Z^5$, and the convergence is uniform in $g \in \mathbb{Z}^5$.
Consequently, since $1= R_1 > \zeta_1(r)^n$, we get that
$$
\bigg |C_0n^{2}\mu_1^{(n)}(g)\zeta_1(r)^n-\frac{1}{\sqrt n}\zeta_1(r)^n\mathrm{e}^{-\frac{\Sigma(g)}{2n}}\bigg|\underset{n\to \infty}{\longrightarrow}0.
$$
Thus, by Karamata's Tauberian theorem \cite[Corollary 1.7.3]{BGT87} we get that as $r\to R$
$$
G_1^{(2)}(e,g|\zeta_1(r))\sim C_g \cdot \sum_{n\geq 1}\frac{1}{\sqrt n}\zeta_1(r)^n\mathrm{e}^{-\frac{\Sigma(g)}{2n}}\sim C'_g (1-\zeta_1(r))^{-1/2},
$$
where $C_g$ and $C_g'$ depend only on $g$. Since $G_1^{(2)}$ is roughly the same as $I_1^{(2)}$, we deduce that
\begin{equation}\label{e:I2sqrt1-zeta}
I_1^{(2)}(e,e|\zeta_1(r))\asymp (1-\zeta_1(r))^{-1/2},
\end{equation}
where the implicit constant is independent of $r$.
More generally, as $r\to R$ we have
\begin{equation}\label{e:I2sqrt1-zetaanyg}
I_1^{(2)}(e,g|\zeta_1(r))\asymp (1-\zeta_1(r))^{-1/2},
\end{equation}
where the implicit constant only depends on $g$.

Now, for any $g,h\in \Gamma_1=\mathbb Z^5$, we have
$$G_1(h,g|\zeta_1(r))\geq C_1 G_1(h,e|\zeta_1(r))G_1(e,g|\zeta_1(r))$$
where the constant $C_1$ is independent of $g,h$ and $r$.
We deduce that
\begin{align*}
I_1^{(2)}(e,g|\zeta_1(r))&\ =\sum_{w_1,w_2\in \Gamma_1}G(e,w_1|\zeta_1(r))G(w_1,w_2|\zeta_2(r))G(w_2,g|\zeta_2(r))\\
&\ \geq C_1 \cdot I_1^{(2)}(e,e|\zeta_1(r))G(e,g|\zeta_1(r)).
\end{align*}
Consequently, letting $k$ be such that $i_k=1$, so that $s_k\in \Gamma_1 = \mathbb Z^5$,
we have
$$I_{1}^{(2)}(e,e|\zeta_{1}(r))\lesssim \frac{I_{1}^{(2)}(e,s_k|\zeta_{1}(r))}{G_{1}(e,s_k|\zeta_1(r))}$$
where the implicit constant is independent of $s_k$ and $r$.
Thus, by equation \eqref{e:I2sqrt1-zeta},
$$(1-\zeta_1(r))^{-1/2}\lesssim \frac{I_{1}^{(2)}(e,s_k|\zeta_{1}(r))}{G_{1}(e,s_k|\zeta_1(r))}.$$
Also, by~(\ref{e:I2sqrt1-zetaanyg}),
$$\frac{I_{1}^{(2)}(e,s_k|\zeta_{1}(r))}{G_{1}(e,s_k|\zeta_1(r))}\lesssim (1-\zeta_1(r))^{-1/2}C_k,$$
where $C_k$ only depends on $s_k$.
We conclude that
\begin{equation}\label{e:timektildephisqrt1-zetaZ5}
1\lesssim (1-\zeta_1(r))^{1/2}\frac{I_{i_k}^{(2)}(e,s_k|\zeta_{i_k}(r))}{G_{i_k}(e,s_k|\zeta_{i_k}(r))}\lesssim C_k.
\end{equation}

On the other hand, the random walk is not spectrally degenerate along $\Gamma_2=\mathbb Z$.
Thus, if $i_k=2$, then
$$\frac{I_{2}^{(2)}(e,s_k|\zeta_{2}(R))}{G_{2}(e,s_k|\zeta_2(R))}$$
is finite, hence
\begin{equation}\label{e:timektildephisqrt1-zetaZ}
    (1-\zeta_1(r))^{1/2}\frac{I_{i_k}^{(2)}(e,s_k|\zeta_{i_k}(r))}{G_{i_k}(e,s_k|\zeta_{i_k}(r))}\underset{r\to R}{\longrightarrow}0
\end{equation}

Finally, let $\xi$ be a conical limit point. In our context of free products, this means that $\xi$ is an infinite word, and $y$ converges to $\xi$ if and only if for every $k$, the prefix of length $k$ of $y$ in its normal form eventually agrees with the prefix of length $k$ of $\xi$.
We deduce that for fixed $x$, the largest common prefix $y_{l-1}$ of $x$ and $y$, as well as its successor $y_l$ only depend on $\xi$ and not on $y$, as soon as $|y|$ is large enough.
Combining~(\ref{e:timektildephisqrt1-zetaZ5}) and~(\ref{e:timektildephisqrt1-zetaZ}), we get
$$(1-\zeta_1(r))^{1/2}\big|\Tilde{\Phi}_r^{(2)}(x,y_{l})-\Tilde{\Phi}_r^{(2)}(e,y_{l})\big|\asymp C(x,\xi),$$
where $C(x,\xi)$ is a finite constant only depending on $x$ and $\xi$.

Also, as $y$ converges to $\xi$, we get that $i_k=1$ occurs infinitely many times, that is $s_k\in \Gamma_1$ for arbitrarily many $k$,
hence by~(\ref{e:I2sqrt1-zeta})
$$\liminf_{r\to R}(1-\zeta_1(r))^{1/2}\Tilde{\Phi}_r^{(2)}(e,y)\underset{|y|\to \infty}{\longrightarrow}\infty.$$
Therefore,
$$\lim_{r\to R}\frac{\Tilde{\Phi}_r^{(2)}(x,y_{l})-\Tilde{\Phi}_r^{(2)}(e,y_{l})}{\Tilde \Phi_r^{(2)}(e,y)}=\lim_{r\to R}\frac{(1-\zeta(r))^{1/2}\big (\Tilde{\Phi}_r^{(2)}(x,y_{l})-\Tilde{\Phi}_r^{(2)}(e,y_{l})\big)}{(1-\zeta(r))^{1/2}\Tilde \Phi_r^{(2)}(e,y)}$$
converges to 0 as $|y|$ tends to infinity, which proves~(\ref{e:conclusionproof7.1}) and thus concludes the proof.
\end{proof}

\begin{Exl} \label{ex:non-home-R-mart}
Consider Example~\ref{ex:limitprobmeasures}, where we constructed two probability measures $\mu$ and $\mu'$ on $\Gamma=\mathbb Z^5*\mathbb Z$ such that $\mu$ is convergent and spectrally degenerate along $\mathbb Z^5$ and $\mu'$ is spectrally non-degenerate. Then the $R$-Martin boundaries $\partial_{M,R}(\Gamma,\mu)$ and $\partial_{M,R}(\Gamma,\mu')$ associated with $\mu$ and $\mu'$ are not homeomorphic. 

Note first that since the parabolic subgroups $\mathbb{Z}^5$ and $\mathbb{Z}$ are virtually abelian, by \cite[Theorem~1.4]{DG21} the $R$-Martin boundaries are minimal, so that $\overline{\partial_{M,R}^m\Gamma}=\partial_{M,R}\Gamma$ for both $\mu$ and $\mu'$.

Now, since $\mathbb Z$ is hyperbolic, we have that $\Gamma$ is also hyperbolic relative (only) to $\mathbb Z^5$. In fact, $\Gamma$ is the HNN-extension of $\mathbb Z^5$ over the trivial subgroup. Following \cite{Bow12}, with this relatively hyperbolic structure the Bowditch boundary consists of the set of infinite words with normal form for the HNN-extension structure glued together with one point at infinity for every left coset of $\mathbb Z^5$.
In particular, as $\mathbb Z^5$ is one-ended, we see that the Bowditch boundary coincides with the set of ends.
Thus, by \cite{Hop44} the Bowditch boundary is homeomorphic to the Cantor set and is in particular totally disconnected.

Now, on the one hand, the measure $\mu$ is spectrally degenerate along $\mathbb Z^5$, so by \cite[Theorem~1.2]{DG21} the $R$-Martin boundary is homeomorphic to the Bowditch boundary, hence it is totally disconnected.
On the other hand, the measure $\mu'$ is spectrally non-degenerate, so \cite[Theorem~1.2]{DG21} shows that the $R$-Martin boundary contains embedded 4-spheres. Therefore, the $R$-Martin boundaries for $\mu$ and $\mu'$ cannot be homeomorphic.
\end{Exl}

Next, we show that the ratio-limit boundaries generally depend on the random walk.

\begin{Exl} \label{ex:non-iso-rl-bndry}
Let $\mu$ and $\mu'$ be probability measures on $\Gamma = \mathbb{Z}^5 * \mathbb{Z}$ as in Example \ref{ex:non-home-R-mart} so that $\mu$ is convergent and spectrally degenerate along $\mathbb{Z}^5$ and $\mu'$ is spectrally non-degenerate. Then the ratio-limit boundaries $\partial_{\rho}(\Gamma,\mu)$ and $\partial_{\rho}(\Gamma,\mu')$ associated with $\mu$ and $\mu'$ are not $\Gamma$-equivariantly homeomorphic.

Indeed, by Proposition~\ref{p:conical-in-rlb} and Proposition~\ref{p:conicalconvergentcase} for $\mu$ and $\mu'$ respectively, we know that
$$\frac{H(x,y)}{K_R(x,y)}\to 1$$
as $y$ converges to a conical limit point. Now, as in the proof of Corollary~\ref{cor:embedding}, we deduce for both $\mu$ and $\mu'$ that there is a bi-Lipschitz $\Gamma$-equivariant map
$$\iota : \partial_{M,R}\Gamma \rightarrow \partial_{\rho} \Gamma$$ which is a homeomorphism onto its closed image, such that for every conical limit point $\xi$ we have
$$H(x,\iota(\xi)) = K_R(x,\xi).$$

The proof of Theorem~\ref{t:essent-min} can then be applied to show that the $R$-Martin boundary (which coincides with the closure of minimal points in the $R$-Martin boundary) is the unique smallest closed $\Gamma$-invariant subspace of the ratio limit boundary. Thus, if the ratio-limit boundaries were $\Gamma$-equivariantly homeomorphic, it would follow that so are their unique smallest closed $\Gamma$-invariant subspaces, in contradiction to Example \ref{ex:non-home-R-mart}.
\end{Exl}

\section{Co-universal quotients of Toeplitz C*-algebras.} \label{s:co-univ}

Our goal in this final section is to identify a unique smallest equivariant quotient of Toeplitz C*-algebra of a random walk, when $\Gamma$ is relatively hyperbolic, and $\mu$ is a finitely supported, admissible, aperiodic and symmetric probability measure on $\Gamma$ which is spectrally non-degenerate.

Throughout this section, we will denote Hilbert spaces by $\mathcal{H}$, and by $\mathbb{B}(\mathcal{H})$ and $\mathbb{K}(\mathcal{H})$ the C*-algebras of all bounded and compact operators on $\mathcal{H}$, respectively. When referring to an ideal $\mathcal{I} \lhd \mathcal{A}$ in a C*-algebra $\mathcal{A}$, we will always mean a two-sided norm-closed ideal.

The following notion of co-universality is the one that fits our context. It is similar in spirit to other such notions that appear in the literature \cite{CLSV11, DK20, Seh22}, particularly \cite[Definition 4.7]{DKKLL22}, but it does not require injectivity on some specified subalgebra.

Recall that if $\mathcal{A}$ be a C*-algebra, a $G$-action $\alpha: G \curvearrowright \mathcal{A}$ is a homomorphism $G \rightarrow \Aut(\mathcal{A})$ such that $\alpha_g$ is a $*$-automorphism, and for every element $a\in \mathcal{A}$ the function $g\mapsto \alpha_g(a)$ is norm continuous. When $\mathcal{I} \lhd \mathcal{A}$ is a $G$-invariant ideal, we get an induced action $\alpha^{\mathcal{I}}: G \curvearrowright \mathcal{A} / \mathcal{I}$ given by $\alpha_g^{\mathcal{I}}(a+\mathcal{I}) = \alpha_g(a) + \mathcal{I}$, which is still point-norm continuous. 

\begin{Def}
Let $\mathcal{A}$ be a C*-algebra, and $G$ a locally compact Hausdorff group. Suppose that $\alpha: G \curvearrowright \mathcal{A}$ is an action by $G$. Let $\mathcal{I}$ be an $G$-invariant ideal in $\mathcal{A}$. We say that $\mathcal{C}:=\mathcal{A} / \mathcal{I}$ is \emph{$G$ co-universal} if for every $*$-representation $\pi : \mathcal{A} \rightarrow \mathbb{B}(\mathcal{H})$ such that
\begin{enumerate}
    \item there is a group action $\beta : G \curvearrowright \mathcal{B}$, where $\mathcal{B}:=\pi(\mathcal{A})$ and;
    \item $\pi : \mathcal{A} \rightarrow \mathcal{B}$ is $G$-equivariant,
\end{enumerate}
there is a $G$-equivariant surjective $*$-homomorphism $\tau_{\pi} : \mathcal{B} \rightarrow\mathcal{C}$.
\end{Def}

The following shows that the above notion of co-universality coincides with the existence of a unique smallest $G$-equivariant quotient of $\mathcal{A}$.

\begin{Prop} \label{p:co-univ-ideal}
Let $\mathcal{A}$ be a C*-algebra, and $G$ a locally compact Hausdorff group, with an action $\alpha : G \curvearrowright \mathcal{A}$. Let $\mathcal{I}$ be a $G$-invariant ideal. Then the quotient $\mathcal{C} := \mathcal{A} / \mathcal{I}$ is $G$ co-universal if and only if $\mathcal{I}$ is the unique largest $G$-invariant ideal of $\mathcal{A}$.
\end{Prop}

\begin{proof}
Suppose first that $\mathcal{C}$ is $G$ co-universal, and let $\mathcal{J} \lhd \mathcal{A}$ be some $G$-invariant ideal. Then the surjective $*$-homomorphism $q_{\mathcal{J}} :\mathcal{A} \rightarrow \mathcal{A} / \mathcal{J}$ is $G$-equivariant with the induced action $\alpha^{\mathcal{J}} : G \curvearrowright \mathcal{A} / \mathcal{J}$ on the image, and by Gelfand-Naimark-Segel theorem we may embed $\mathcal{A} / \mathcal{J}$ as a norm-closed subalgebra of $\mathcal{B}(\mathcal{H})$ for some Hilbert space $\mathcal{H}$, making $q_{\mathcal{J}}$ into a genuine $*$-representation. By $G$ co-universality of $\mathcal{C}$ we have a $G$-equivariant surjective $*$-homomorphism $\tau_{\mathcal{J}} : \mathcal{A}/\mathcal{J} \rightarrow \mathcal{A}/\mathcal{I}$, which implies that $\mathcal{J}\subseteq \mathcal{I}$. Hence, $\mathcal{I}$ is the unique largest $G$-invariant ideal of $\mathcal{A}$.

Conversely, if $\mathcal{I}$ is the largest $G$-invariant ideal of $\mathcal{A}$, and $\pi : \mathcal{A} \rightarrow \mathbb{B}(\mathcal{H})$ a $G$-equivariant $*$-representation such that $\mathcal{B}:=\pi(\mathcal{A})$ admits a $G$-action $\beta:G\curvearrowright \mathcal{B}$, then the kernel $\ker \pi$ is a $G$-invariant ideal, and hence $\ker \pi \lhd \mathcal{I}$. By applying the first isomorphism theorem for C*-algebras to $\pi$, by composition we obtain a $G$-equivariant map $\tau_{\pi}: \mathcal{B} \cong \mathcal{A}/\ker \pi \rightarrow \mathcal{A} / \mathcal{I}$.
\end{proof}

Thus, in order to show that with respect to a group action $G \curvearrowright \mathcal{A}$ a $G$ co-universal quotient exists, it will suffice to show that there is a \emph{unique} smallest $G$-equivariant quotient $\mathcal{C}$. The following C*-algebras associated with Markov chains were studied by the first author and his collaborators in a sequence of papers \cite{DOM14, DOM16, CDHLZ21, DO21}. These works form part of a larger program to determine the structure of natural operator algebras arising from subproduct systems, pioneered by Shalit, Solel and Viselter \cite{SS09, Vis11, Vis12}.

Let $P$ be the transition kernel of a random walk on a discrete group $\Gamma$ induced by an admissible probability measure $\mu$ on $\Gamma$. Coarse analogues of Toeplitz C*-algebras for random walks and their subaglebras were first studied in \cite{DOM14, DOM16} (see also \cite{CDHLZ21}), and the following standard version of Toeplitz C*-algebras for random walks was defined and studied in \cite{DO21}. For $m\in \mathbb{N}$ we denote by $E(P^m) := \{ \ (x,y) \ | \ P^m(x,y) > 0 \ \}$.
The Toeplitz C*-algebras that we consider are norm closed $*$-subalgebras of bounded operators on the Hilbert space
$$
\mathcal{H}_P := \oplus_{m=0}^{\infty} \ell^2 (E(P^m))
$$
with the standard orthonormal basis $\{e_{y,z}^{(m)}\}_{(y,z)\in E(P^m)}$. Then, for any $n\in \mathbb{N}$ and $(x,y) \in E(P^n)$ we define an operator $S_{x,y}^{(n)}$ on $\mathcal{H}_P$ by specifying for $(y',z) \in E(P^m)$
$$
S_{x,y}^{(n)}(e_{y',z}^{(m)}) = \delta_{y,y'} \sqrt{\frac{P^n(x,y)P^m(y,z)}{P^{n+m}(x,z)}} e_{x,z}^{(n+m)}.
$$
For a fixed $z\in \Gamma$ we also denote $\mathcal{H}_{P,z} = \overline{\Span} \{ \ e_{y,z}^{(m)} \ | \ (y,z) \in E(P^m) \}$, which is a reducing subspace for all of the operators $S_{x,y}^{(n)}$ defined above. Thus, we see that our operators are all in the direct product $\prod_{z\in \Gamma} \mathbb{B}(\mathcal{H}_{P,z})$.

\begin{Def}
Let $P$ be a random walk on a discrete group $\Gamma$ induced by a probability measure $\mu$. The \emph{Toeplitz C*-algebra} of $(\Gamma,\mu)$ is given by
$$
\mathcal{T}(\Gamma,\mu):= C^*( \ S_{x,y}^{(n)} \ | \ (x,y) \in E(P^n), \ n\in \mathbb{N} \ ).
$$
The \emph{Cuntz C*-algebra} of $(\Gamma,\mu)$ is given by
$$
\mathcal{O}(\Gamma,\mu):= \mathcal{T}(\Gamma,\mu) / \mathcal{J}(\Gamma,\mu),
$$
where $\mathcal{J}(\Gamma,\mu) = \mathcal{T}(\Gamma,\mu) \cap \prod_{z\in \Gamma} \mathbb{K}(H_{P,z})$.
\end{Def}

These C*-algebras come equipped with natural $\Gamma \times \mathbb{T}$ actions. Namely, we have a natural unit circle action $\gamma : \mathbb{T} \curvearrowright \mathcal{T}(\Gamma,\mu)$ given by $\gamma_z(S^{(n)}_{x,y}) = z^n S^{(n)}_{x,y}$, as well as a $\Gamma$-action $\delta : \Gamma \curvearrowright \mathcal{T}(\Gamma,\mu)$ given by $\delta_g(S^{(n)}_{x,y}) = S^{(n)}_{gx,gy}$. As these actions commute, we get an action $\lambda: \Gamma\times \mathbb{T} \curvearrowright \mathcal{T}(\Gamma,\mu)$, and since $\prod_{z\in \Gamma} \mathbb{K}(H_{P,z})$ is $\Gamma\times \mathbb{T}$-invariant, we get an induced action $\overline{\lambda} : \Gamma \times \mathbb{T} \curvearrowright \mathcal{O}(\Gamma,\mu)$ which acts similarly on generators.

Recall that SRLP stands for the strong ratio limit property.
One of the main results of \cite{DO21} is the computation of $\mathcal{O}(\Gamma,\mu)$ under the assumption of SRLP. Namely, in \cite[Theorem 4.10]{DO21} it was shown that by assuming SRLP, we get $\mathcal{O}(\Gamma,\mu) \cong C(\Delta_{\rho}^r \Gamma \times \mathbb{T}) \otimes \mathbb{K}(\ell^2(\Gamma))$. Under this identification (See \cite[Section 5]{DO21}), the induced action $\overline{\lambda} : \Gamma \times \mathbb{T} \curvearrowright \mathcal{O}(\Gamma,\mu)$ is given for $f \otimes K \in C(\Delta_{\rho}^r \Gamma \times \mathbb{T}) \otimes \mathbb{K}(\ell^2(\Gamma))$ by
$$
\overline{\lambda}_{z,g}(f \otimes K) =((g,z) \cdot f) \otimes S_gKS_g^{-1}, 
$$
where $S : \Gamma \rightarrow \mathbb B(\ell^2(\Gamma))$ is the left regular representation of $\Gamma$, and for $(\alpha,w) \in \Delta_{\rho}^r \Gamma \times \mathbb{T}$, $((g,z) \cdot f)(\alpha,w) = f(g^{-1}\alpha,zw)$. 

Now, suppose $\Gamma$ is non-elementary hyperoblic relative to a finite collection of subgroups $\Omega$, and $\mu$ is a finitely supported, admissible, aperiodic and symmetric probability measure on $\Gamma$. By combining Corollary \ref{cor:embedding} and Corollary \ref{cor:reduced} we obtain a $\Gamma$-equivariant injection $\iota : \overline{\partial^m_{M,R} \Gamma} \rightarrow \partial_{\rho}^r \Gamma$. Let $q_P : \mathcal{T}(\Gamma,\mu) \rightarrow \mathcal{O}(\Gamma,\mu)$ be the canonical surjection, and define the $\lambda$ invariant ideal $\mathcal{J}_{\lambda}$ of $\mathcal{T}(\Gamma,\mu)$ by setting
$$
\mathcal{J}_{\lambda} := q_P^{-1} \Big[ C([\Delta_{\rho}^r \Gamma \setminus \iota(\overline{\partial^m_{M,R} \Gamma})] \times \mathbb{T}) \otimes \mathbb{K}(\ell^2(\Gamma)) \Big].
$$

The following result extends \cite[Corollary 5.2]{DO21} to relatively hyperbolic groups with symmetric, aperiodic,  spectrally non-degenerate random walks. This partly answers the first part of \cite[Question 5.6]{DO21} for a large class of random walks on relatively hyperbolic groups in the form of co-universality.

\begin{Thm}\label{t:co-universalquotient}
Suppose $\Gamma$ is non-elementary hyperoblic relative to a finite collection of subgroups $\Omega$, and $\mu$ is aperiodic, symmetric, spectrally non-degenerate probability measure on $\Gamma$ with finite support generating $\Gamma$. Then $\mathcal{T}(\Gamma,\mu)$ admits a $\Gamma\times \mathbb{T}$ co-universal quotient which coincides with
$$
\mathcal{T}(\Gamma,\mu) / \mathcal{J}_{\lambda} \cong C(\overline{\partial^m_{M,R} \Gamma} \times \mathbb{T}) \otimes \mathbb{K}(\ell^2(\Gamma)).
$$
\end{Thm}

\begin{proof}
By Proposition~\ref{p:co-univ-ideal} it will suffice to show that $\mathcal{J}_{\lambda}$ is the largest $\Gamma \times \mathbb{T}$-invariant ideal of $\mathcal{T}(\Gamma,\mu)$. Let $\mathcal{J}$ be a proper $\Gamma \times \mathbb{T}$-invariant ideal of $\mathcal{T}(\Gamma,\mu)$, and denote by $\overline{\mathcal{J}} := q_P(\mathcal{J})$. Then there are two possibilities. 

If $\mathcal{J}(\Gamma,\mu) \subseteq \mathcal{J}$, then $\overline{\mathcal{J}}$ is a proper ideal of $\mathcal{O}(\Gamma,\mu) \cong C(\Delta_{\rho}^r \Gamma \times \mathbb{T}) \otimes \mathbb{K}(\ell^2(\Gamma))$. Thus, by Gelfand-Naimark duality, there exists an open $\Gamma \times \mathbb{T}$-invariant subset $Y \subset \Delta_{\rho}^r \Gamma \times \mathbb{T}$ such that $\overline{\mathcal{J}} = C(Y) \otimes \mathbb{K}(\ell^2(\Gamma))$. By Theorem~\ref{t:essent-min} we know that $\iota(\overline{\partial^m_{M,R} \Gamma}) \cong \overline{\partial^m_{M,R} \Gamma}$ is the unique smallest closed $\Gamma$-invariant subspace of $\partial_{\rho} \Gamma \cong \partial^r_{\rho} \Gamma$. Since the action of $\Gamma$ on $\Gamma/R_{\mu}$ is transitive, we see that $\Gamma/R_{\mu}$ has no closed $\Gamma$-invariant subsets, so that any proper closed $\Gamma$-invariant subset of $\Delta_{\rho}^r \Gamma$ must be a subset of $\partial^r_{\rho} \Gamma$. Thus, $\iota(\overline{\partial^m_{M,R} \Gamma})$ is the unique smallest closed $\Gamma$-invariant subspace of $\Delta_{\rho}^r \Gamma$.

Since $\mathbb{T} \curvearrowright \mathbb{T}$ acts minimally, we get that $\iota(\overline{\partial^m_{M,R} \Gamma}) \times \mathbb{T}$ is the unique smallest closed $\Gamma \times \mathbb{T}$-invariant subspace of $\Delta_{\rho}^r \Gamma \times \mathbb{T}$. Since $Y$ is proper, open and $\Gamma\times \mathbb{T}$ invariant, its complement in $\Delta_{\rho}^r \Gamma \times \mathbb{T}$ must contain $\iota(\overline{\partial^m_{M,R} \Gamma}) 
\times \mathbb{T}$. Thus, $Y$ is contained in $[\Delta_{\rho}^r \Gamma \setminus \iota(\overline{\partial^m_{M,R} \Gamma})] \times \mathbb{T}$, and we get that $\mathcal{J} \subseteq \mathcal{J}_{\lambda}$.

Now suppose that $\mathcal{J}$ is a general $\Gamma \times \mathbb{T}$-invariant ideal. Since $\overline{\mathcal{J}}$ is proper in $\mathcal{O}(\Gamma,\mu)$, we get that $\mathcal{J} + \mathcal{J}(\Gamma,\mu)$ is also proper in $\mathcal{T}(\Gamma,\mu)$. Hence, by the previous argument we get that $\mathcal{J} \subseteq \mathcal{J} + \mathcal{J}(\Gamma,\mu) \subseteq \mathcal{J}_{\lambda}$.
\end{proof}

Although it is unknown if all relatively hyperbolic groups admit a spectrally non-degenerate random walk, by Propositions \ref{propconstructionnonspecdeg} and \ref{propsmallhomogeneousdimension}, there are many spectrally non-degenerate adapted random walks on arbitrary free products, and any random walk on relatively hyperbolic groups with virtually nilpotent subgroups of homogeneous dimension at most $4$ is spectrally non-degenerate. Hence, beyond the class of hyperbolic groups, there is a wide class of relatively hyperbolic groups for which there always exists a random walk such that $\mathcal{T}(\Gamma,\mu)$ has a $\Gamma \times \mathbb{T}$ co-universal quotient. 

As a final application, we use of co-universality for two random walks on $\Gamma =\mathbb{Z}^5 *\mathbb{Z}$ as in Example \ref{ex:non-home-R-mart} to come up with examples of Toeplitz C*-algebras $\mathcal{T}(\Gamma,\mu)$ and $\mathcal{T}(\Gamma,\mu')$ for which no two non-trivial $\Gamma \times \mathbb{T}$-equivariant quotients are isomorphic.

\begin{Exl} \label{ex:no-gamma-t-quotient-iso}
Let $\mu$ and $\mu'$ be probability measures on $\Gamma = \mathbb{Z}^5 * \mathbb{Z}$ as in Example \ref{ex:non-home-R-mart} so that $\mu$ is convergent and spectrally degenerate along $\mathbb{Z}^5$ and $\mu'$ is spectrally non-degenerate. Suppose that $\mathcal{C}_{\mu}$ and $\mathcal{C}_{\mu'}$ are non-trivial $\Gamma \times \mathbb{T}$-equivariant quotients of $\mathcal{T}(\Gamma,\mu)$ and $\mathcal{T}(\Gamma,\mu')$ respectively. Then $\mathcal{C}_{\mu}$ and $\mathcal{C}_{\mu'}$ cannot be isomorphic.
In particular, the Toeplitz C*-algebras $\mathcal{T}(\Gamma,\mu)$ and $\mathcal{T}(\Gamma,\mu')$ as well as the Cuntz C*-algebras $\mathcal{O}(\Gamma,\mu)$ and $\mathcal{O}(\Gamma,\mu)$ are not isomorphic.

Recall first that for both $\mu$ and $\mu'$, the closures of minimal points in their $R$-Martin boundaries coincide with their whole $R$-Martin boundary. Suppose now that the quotients $C_\mu$ and $C_{\mu'}$ are $*$-isomorphic via some $*$-isomorphism $\varphi$. Then, by the same proof of Theorem \ref{t:co-universalquotient} for $\mu$ (using the same proof of Theorem \ref{t:essent-min} by appealing to Proposition \ref{p:conicalconvergentcase} instead of Proposition \ref{p:conical-in-rlb}) as well as Theorem \ref{t:co-universalquotient} for $\mu'$ there exist \emph{unique} smallest $\Gamma \times \mathbb{T}$-equiavariant quotients $q : \mathcal{C}_{\mu} \rightarrow C(\partial_{M,R} (\Gamma,\mu) \times \mathbb{T}) \otimes \mathbb{K}(\ell^2(\Gamma))$ and $q': \mathcal{C}_{\mu'} \rightarrow C(\partial_{M,R} (\Gamma,\mu') \times \mathbb{T}) \otimes \mathbb{K}(\ell^2(\Gamma))$. Since these quotients are unique, the ideals $\ker q$ and $\ker q'$ are the unique largest proper $\Gamma \times \mathbb{T}$ invariant ideals in their repsective Toeplitz algebras. Therefore, as $\varphi$ is $\Gamma \times \mathbb{T}$-equivariant, it must map $\ker q$ onto $\ker q'$. Thus, we obtain an induced $*$-isomorphism $\widetilde{\varphi} : C(\partial_{M,R} (\Gamma,\mu) \times \mathbb{T}) \otimes \mathbb{K}(\ell^2(\Gamma)) \rightarrow C(\partial_{M,R} (\Gamma,\mu') \times \mathbb{T}) \otimes \mathbb{K}(\ell^2(\Gamma))$.

By Gelfand-Naimark duality, we must then have that $\partial_{M,R} (\Gamma,\mu) \times \mathbb{T}$ and $\partial_{M,R} (\Gamma,\mu') \times \mathbb{T}$ are homeomorphic. However, in Example \ref{ex:non-home-R-mart} we saw that $\partial_{M,R} (\Gamma,\mu)$ has topological dimension $0$ while $\partial_{M,R} (\Gamma,\mu')$ has topological dimension at least $4$. Hence, we arrive at a contradiction because $\partial_{M,R} (\Gamma,\mu) \times \mathbb{T}$ is at most $1$ dimensional while $\partial_{M,R} (\Gamma,\mu') \times \mathbb{T}$ is at least $4$ dimensional.
\end{Exl}

Let us conclude with a few words about possible ways to generalize our work.
According to Corollary~\ref{c:ratio-limit-iterated-sums}~(2), whenever $\Gamma$ is hyperbolic relative to virtually abelian subgroups, for any finitely supported, admissible, aperiodic, symmetric and convergent random walk, the ratio limit kernel $H(x,y)$ can be expressed as the limit of $I^{(s)}(x,y|r)/I^{(s)}(e,y|r)$ as $r\to R$.
In this situation, beyond the specific case of Example~\ref{ex:no-gamma-t-quotient-iso}, the proof of Proposition~\ref{p:conical-in-rlb} might be adaptable to show that once again that for every $\xi$ in the Martin boundary such that $\pi(\xi)$ is conical and for every sequence $y_n$ converging to $\xi$ we still have that
$H(\cdot,y_n)/K_R(\cdot,y_n)$ converges to 1.
Consequently, the conclusions of Theorem~\ref{t:essent-min} and Theorem~\ref{t:co-universalquotient} might still hold in this situation.
This raises the question how much the assumption of being spectrally non-degenerate can be relaxed for our main results to hold.

However, as illustrated in Example~\ref{exampleallderivativefinite}, we know that there are limitations to the overall strategy we used. Hence, new techniques are necessary to prove a result analogous to Theorem~\ref{t:co-universalquotient} for any admissible, admissible, aperiodic and symmetric random walk on arbitrary relatively hyperbolic group.


\end{document}